\documentclass{amsart}
\usepackage{amsmath, amssymb, amsfonts}
\usepackage{bm}
\usepackage{mathtools}
\usepackage[hidelinks]{hyperref}
\usepackage{cleveref}
\usepackage{lineno}
\usepackage{subfig}

\newtheorem{theorem}{Theorem}[section]
\newtheorem{lemma}[theorem]{Lemma}
\newtheorem{assumption}[theorem]{Assumption}

\theoremstyle{definition}

\theoremstyle{remark}
\newtheorem{remark}[theorem]{Remark}

\numberwithin{equation}{section}

\crefname{lemma}{Lemma}{Lemmata}
\crefname{theorem}{Theorem}{Theorems}
\crefname{corollary}{Corollary}{Corollarys}
\crefname{algorithm}{Algorithm}{Algorithms}
\crefname{assumption}{Assumption}{Assumptions}
\crefname{definition}{Definition}{Definitions}
\crefname{proof}{Proof}{Proofs}
\crefname{proof}{Proof}{Proofs}
\crefname{figure}{Figure}{Figures}
\crefname{section}{Section}{Sections}
\crefname{subsection}{Subsection}{Subsections}
\crefname{equation}{}{}

\DeclareMathOperator{\dist}{dist}
\DeclareMathOperator{\supp}{supp}
\DeclareMathOperator{\DG}{DG}
\DeclareMathOperator{\Div}{div}
\DeclareMathOperator{\DGbm}{\mathbf{DG}}
\let\P\relax
\DeclareMathOperator{\P}{P}
\let\d\relax
\DeclareMathOperator{\d}{d\!}

\begin{document}

\allowdisplaybreaks

\title[Numerical Analysis for Hyperbolic PDE-Constrained Optimization]{Numerical Analysis for a Hyperbolic PDE-Constrained Optimization Problem in Acoustic Full Waveform Inversion}

\author{Luis Ammann}
\address{University of Duisburg-Essen, Fakult\"at f\"ur Mathematik, Thea-Leymann-Str. 9, D-45127 Essen, Germany.}
\curraddr{}
\email{luis.ammann@uni-due.de}
\thanks{This work is supported by the DFG research grants YO159/2-2 and YO159/5-1.}

\author{Irwin Yousept}
\address{University of Duisburg-Essen, Fakult\"at f\"ur Mathematik, Thea-Leymann-Str. 9, D-45127 Essen, Germany.}
\curraddr{}
\email{irwin.yousept@uni-due.de}
\thanks{}



\date{\today}

\begin{abstract}
This paper explores a fully discrete approximation for a nonlinear hyperbolic PDE-constrained optimization problem (P) with applications in acoustic full waveform inversion. The optimization problem is primarily complicated by the hyperbolic character and the second-order bilinear structure in the governing wave equation. While the control parameter is discretized using the piecewise constant elements, the state discretization is realized through an auxiliary first-order system along with the leapfrog time-stepping method and continuous piecewise linear elements. The resulting fully discrete minimization problem ($\text{P}_h$) is shown to be well-defined. Furthermore, building upon a suitable {CFL}-condition, we prove stability and uniform convergence of the state discretization. Our final result is the strong convergence result for ($\text{P}_h$) in the following sense: Given a local minimizer $\overline \nu$ of (P) satisfying a reasonable growth condition, there exists a sequence of local minimizers of ($\text{P}_h$) converging strongly towards $\overline \nu$.
\end{abstract}

\maketitle

\section{Introduction}
	
\label{section:introduction}

This paper proposes and analyzes a fully discrete approximation for a hyperbolic PDE-constrained optimization problem arising in acoustic full waveform inversion (FWI). Let us begin by formulating the corresponding state equation given by the damped acoustic second-order wave equation: 
\begin{equation}
	\label{system intro}
	\left\{\begin{aligned}
		&\nu\partial_{tt}p - \Delta p + \eta \partial_t p = f 
		&&\text{in }I\times \Omega
		\\
		&\partial_n p = 0 
		&&\text{on }I\times\Gamma_N
		\\
		&p = 0 
		&&\text{on }I\times\Gamma_D
		\\
		&(p,\partial_t p)(0) = (p_0,p_1) 
		&&\text{in }\Omega,
	\end{aligned}\right.
\end{equation}
where $\Omega \subset \mathbb R^N$ ($N\geq 2$) denotes a bounded Lipschitz domain with $\partial \Omega = \Gamma_N \cup \Gamma_D$, $I = [0, T]\subset\mathbb R$ the operating finite time interval, $p\colon I\times\Omega\to\mathbb R$ the acoustic pressure, $\nu\colon\Omega\to\mathbb R$ the square slowness parameter, $\eta\colon\Omega\to\mathbb R$ the damping parameter, $f\colon\Omega\to\mathbb R$ the source term, and $p_0, p_1\colon\Omega\to\mathbb R$ the initial data. The precise mathematical assumptions for all quantities involved in \eqref{system intro} are stated in \cref{assumption}. Given some observation data $p^{ob}_{i}\colon I\times\Omega\to\mathbb R$ at receivers modelled through the weight functions $a_i\colon I\times\Omega\to\mathbb R$ for $i=1,\dots,m$, FWI aims at recovering the acoustic parameter $\nu$ by solving the following PDE-constrained optimization problem:
\begin{equation}\label{P intro}
	\left\{\begin{aligned}
		& \inf \mathcal{J}(\nu, p) \coloneqq\frac{1}{2} \sum_{i=1}^m \int_I\int_{\Omega}a_i(p-p^{ob}_i)^2\d x\d t+\frac{\lambda}{2}\|\nu\|_{L^2(\Omega)}^2
		\\
		&\text{ s.t. } \cref{system intro} \text{ and } \nu\in\mathcal V_{ad}\coloneqq\{\nu\in L^2(\Omega): \nu_- \leq \nu(x)\leq \nu_+\text{ for a.e. }x\in\Omega\}.
	\end{aligned}\right.
\end{equation}
For a more detailed discussion on the appearing quantities, particularly their physical explanation, see our previous contribution \cite{ammann23}, that analyzed \cref{P intro} and established its first- and second-order optimality conditions making use of an auxiliary first-order approach (see \cref{system state}). 

While the numerical analysis of optimal control problems governed by elliptic and parabolic PDEs seems to have been extensively explored in the literature, there are only limited works devoted to the hyperbolic ones. We refer to Hou \cite{hou13} and the recent work by Peralta and Kunisch \cite{peralta22} for contributions to this research direction.

This paper is a continuation of \cite{ammann23} focusing on the numerical analysis of \cref{P intro} through a suitable fully discrete approximation. More specifically, we propose a state discretization (see \cref{leapfrog scheme}) based on the first-order formulation \cref{system state} and the leapfrog (Yee) time-stepping method together with piecewise linear elements. The proposed state discretization is shown to be stable under a suitable {CFL}-condition (see \cref{theorem stability}). Then, along with the piecewise constant control discretization, the fully discrete approximation for \eqref{P intro} is obtained in \cref{P discretized reduced}, which turns out to admit at least one global solution (see \cref{theorem existence}). Towards the convergence analysis of \cref{P discretized reduced}, we define specific interpolations associated with \cref{leapfrog scheme} and show their uniform convergence to the solution of \cref{system state} (see \cref{theorem interpolations}). This paves the way for our main convergence result (\cref{theorem convergence}): Given a local minimizer $\overline \nu $ of \eqref{P intro} satisfying a reasonable growth condition, we prove the existence of a sequence of locally optimal solutions of \cref{P discretized reduced} converging strongly 
towards $\overline \nu$. In the final part of this paper, a numerical test is included to illustrate the performance of the proposed finite element discretization.

\section{Preliminaries}
\label{section optimal control}

This section prepares the numerical analysis for \cref{P intro} by recapitulating the well-posedness result from \cite{ammann23}. We denote the space of all equivalence classes of measurable and Lebesgue square integrable $\mathbb R$-valued (resp. $\mathbb R^N$-valued) functions by $L^2(\Omega)$ (resp. $\bm{L}^2(\Omega)$). The space $C^\infty_0(\Omega)$ (resp. $\bm C^\infty_0(\Omega)$) contains all infinitely differentiable $\mathbb R$-valued (resp. $\mathbb R^N$-valued) functions with compact support in $\Omega$. 
Furthermore, we introduce 
\begin{align*}
	C^\infty_D(\Omega)
	&\coloneqq \{v\vert_\Omega : v\in C^\infty(\mathbb R^N), \dist(\supp(v), \Gamma_D)>0\}
	\\
	H^1_D(\Omega)
	&\coloneqq \overline{C^\infty_D(\Omega)}^{\|\cdot\|_{H^1(\Omega)}}
	\\
	\bm{H}_N(\Div,\Omega)
	&\coloneqq\{\bm{u}\in \bm{H}(\Div,\Omega)\colon(\Div\bm{u},\phi)_{L^2(\Omega)}= - (\bm{u},\nabla\phi)_{\bm{L}^2(\Omega)}\,\forall \phi\in H^1_D(\Omega)\}.
\end{align*}

\begin{remark}
	Thanks to the Lipschitz assumption on the domain $\Omega$, an equivalent characterization of $H^1_D(\Omega)$ using the trace operator $\tau\colon H^1(\Omega)\to H^{1/2}(\partial\Omega)$ reads $H^1_D(\Omega) = \{u\in H^1(\Omega): \tau u= 0 \text{ a.e. on }\Gamma_D\}$.
	Then, using the normal trace operator $\gamma_n\colon\bm H(\Div, \Omega)\to H^{-1/2}(\partial\Omega)$, an equivalent characterization of $\bm H_N(\Div, \Omega)$ is obtained as follows:
	\begin{align*}
		\bm H_N(\Div, \Omega)=\{
			&\bm u\in\bm H(\Div, \Omega): \langle \gamma_n\bm u, \phi\rangle_{H^{-1/2}(\partial\Omega), H^{1/2}(\partial\Omega)} = 0
			\\
			&\text{for all }\phi\in H^{1/2}(\partial\Omega)\text{ satisfying }\phi=0\text{ a.e. on }\Gamma_D\}.
	\end{align*}
\end{remark}

Let us now formulate the standing assumption of this paper:
\begin{assumption}\label{assumption}
	Let $\Omega\subset\mathbb R^N$ be a bounded polyhedral Lipschitz domain. The boundary is given by $\partial\Omega = \Gamma_N\cup\Gamma_D$ with the closed subset $\Gamma_D\subsetneq\partial\Omega$ and $\Gamma_N = \partial\Omega\setminus\Gamma_D$ satisfying $|\Gamma_N|\neq 0$. Let $p_0 \in H_D^1(\Omega)$, $p_1 \in L^\infty(\Omega)$, $f\in L^1(I,L^2(\Omega))$, and $p^{ob}_{i} \in W^{1,1}(I,L^2(\Omega))$ for all $i=1,\ldots, m$ and some $m\in\mathbb N$ be given data. The coefficients $a_i\in C^1(I, L^\infty(\Omega))$ and $\eta \in L^\infty(\Omega)$ are also given data and assumed to be nonnegative for all $i=1,\ldots, m$. Finally, let $\nu_-,\nu_+\in\mathbb R$ satisfy $0<\nu_-\leq\nu_+$.
\end{assumption}

We introduce $F \in W^{1,1}(I, L^2(\Omega))$ by $F(t) \coloneqq \int_{0}^{t}f(s)\d s$ for all $t\in I$ and the mapping 
\begin{align}\label{definition Phi}
	&\Phi: L^\infty(\Omega) \to \bm{H}_N(\Div,\Omega), \quad \nu \mapsto\nabla y
 \\\notag
 &\text{where}\quad \int_{\Omega}\nabla y\cdot\nabla \phi\d x = \int_{\Omega} (\eta p_0 + \nu p_1) \phi\d x \quad \forall \phi\in H^1_D(\Omega).
\end{align}
Note that the variational problem in \cref{definition Phi} admits for every $\nu\in L^\infty(\Omega)$ a unique solution $y \in H^1_D(\Omega)$ due to the Lax-Milgram lemma. Furthermore, by the definition of $\bm H_N(\Div, \Omega)$, the above variational inequality immediatly implies that $\nabla y\in\bm H_N(\Div,\Omega)$ with $\Div(\nabla y) = - \nabla p_0 - \nu p_1$. This ensures that $\Phi: L^\infty(\Omega) \to \bm{H}_N(\Div,\Omega)$ is well-defined. In \cite{ammann23}, we have analyzed \cref{P intro} making use of the following auxiliary first-order system:
\begin{equation}\label{system state}
	\left\{\begin{aligned}
		&\nu\partial_tp + \Div\bm{u} + \eta p = F
		&&\text{in }I\times \Omega
		\\
		&\partial_t\bm{u} + \nabla p = \bm{0} 
		&&\text{in }I\times \Omega
		\\
		&p = 0 
		&&\text{on }I\times \Gamma_D
		\\
		&\bm{u} \cdot\bm{n} = 0 
		&&\text{on }I\times \Gamma_N
		\\
		&(p,\bm{u})(0) = (p_0,\Phi(\nu)) 
		&&\text{in }\Omega.
	\end{aligned}\right.
\end{equation}
As shown in \cite[Theorem 2.3]{ammann23}, the auxiliary first-order system \cref{system state} admits for every $\nu \in\mathcal{V}_{ad}$ a unique (classical) solution $(p,\bm u)\in C^1(I, L^2(\Omega)\times\bm L^2(\Omega))\cap C(I, H^1_D(\Omega)\times \bm H_N(\Div,\Omega))$, and the first solution component $p$ is the unique mild solution to the forward system \cref{system intro} in the sense of \cite[Defintion 1.2]{ammann23}. In the following, $S_p\colon\mathcal V_{ad} \to C^1(I, L^2(\Omega))\cap C(I, H^1_D(\Omega)), \nu \mapsto p$, denotes the corresponding solution mapping, assigning to every parameter $\nu$ the first solution component $p$ of the unique solution to \cref{system state}. Making use of this operator, \cref{P intro} can be equivalently written in the reduced formulation as follows:
\begin{equation}\tag{P}
	\label{P}
	\min_{\nu\in \mathcal{V}_{ad}} J(\nu) \coloneqq \mathcal{J}(\nu, S_p(\nu)) = 
	\frac{1}{2} \sum_{i=1}^m \int_I\int_{\Omega}a_i(S_p(\nu)-p^{ob}_i)^2\d x\d t+\frac{\lambda}{2}\|\nu\|_{L^2(\Omega)}^2.
\end{equation}

\section{\texorpdfstring{Fully Discrete Scheme for \cref{system state}}{}}
\label{section fully discrete}

This section is devoted to the construction and analysis of a fully discrete scheme for the governing state system \cref{system intro} based on the first-order approach \cref{system state}. In the following, let $\{\mathcal T_h\}_{h>0}$ denote a quasi-uniform family of triangulations for $\Omega$ such that every edge $E$ of every element $T\in \mathcal T_h$ satisfying $E\cap\partial\Omega\neq\emptyset$ belongs either to 
$\Gamma_D$ or $\overline{\Gamma_N}$, i.e., either $E\subset\Gamma_D$ or $E\subset\overline{\Gamma_N}$. Here, $h>0$ denotes the largest diameter of all elements $K\in \mathcal T_h$. Furthermore, $\text P_1^h$ stands for the space of all continuous piecewise linear functions associated with $\mathcal T_h$, i.e.,
\begin{equation*}
	\text P_1^h\coloneqq\{p\in C(\overline\Omega):\, p\vert_{K}\in\mathbb P_1(K)\quad\forall K\in\mathcal T_h\}.
\end{equation*}
The space of scalar-valued (resp. vector-valued) piecewise constant functions associated with $\mathcal T_h$ is denoted by $\DG_0^h$ (resp. $\DGbm_0^h$). Incorporating the mixed boundary condition in \cref{system state}, we introduce the subspace
\begin{equation*}
	\P_{1,D}^h\coloneqq \{\phi_h\in \text P_1^h:\, \phi_h = 0\,\,\text{on }\Gamma_D\} = \text P_1^h\cap H^1_D(\Omega),
\end{equation*}
where the latter identity follows from the characterization of $H^1_D(\Omega)$ in \cite[Theorem 2.1]{egert17}. Towards discretizing the system \cref{system state} in time, given some $N\in\mathbb N$, we choose the equidistant discretization 
\begin{equation*}
	0=t_0\leq t_1\leq t_2\leq\dots\leq t_N=T,
\end{equation*}
where $t_l-t_{l-1}=\tau=\frac{1}{N}$ for all $l=1,\dots,N$. Furthermore, let us employ the middle points of the constructed subintervals in the time discretization, i.e., $t_{l+ \frac{1}{2}}\coloneqq t_l +\frac{\tau}{2}$ for all $l=0,\dots, N-1$. We aim to construct an approximation for \cref{system state}, motivated by the leapfrog (Yee) time-stepping \cite{yee66}, where we evaluate the first equation in \cref{system state} at the discretization nodes and the second one at the middle points. This approach leads to approximations $\{p_h^l\}_{l=0}^N\subset\P_{1,D}^h$ and $\{\bm u_h^{l+\frac{1}{2}}\}_{l=0}^{N-1}\subset\DGbm_0^h$ where $p_h^l\approx p(t_l)$ and $\bm u_h^{l+\frac{1}{2}}\approx\bm u(t_{l+\frac{1}{2}})$ for the solution $(p, \bm u)$ to \cref{system state}. To arrive at a concise formulation, we introduce the notation
\begin{equation}\label{def deltaphk}
	\delta p_h^{l+\frac{1}{2}}\coloneqq \frac{p_h^{l+1} - p_h^l}{\tau},\quad p_h^{l+\frac{1}{2}}\coloneqq \frac{p_h^{l+1} + p_h^l}{2}\quad\forall l=0,\dots,N-1
\end{equation}
and 
\begin{equation}\label{def deltauhk}
	\delta\bm u_h^l\coloneqq\frac{\bm u_h^{l+\frac{1}{2}} - \bm u_h^{l-\frac{1}{2}}}{\tau}\quad\forall l=1,\dots,N-1.
\end{equation}
Furthermore, let $F^{l+\frac{1}{2}}\coloneqq F(t_{l+\frac{1}{2}})$ for all $l = 0,\dots, N-1$. We denote the $\P^h_1$-interpolation operator by
\begin{equation}\label{definition standard interpolation}
	\mathcal I_h\colon C(\overline\Omega)\to\P^h_1,\quad v\mapsto\sum_{j=1}^{M_h}v(x_j)\phi_{j},
\end{equation}
where $\{\phi_{j}\}_{j=1}^{M_h}\subset\P^h_1$ denotes the standard nodal basis of $\P^h_1$  associated with the corresponding nodal points $\{x_j\}_{j=1}^{M_h}\subset\overline\Omega$. The operator $\mathcal I_h$ satisfies $\mathcal I_h(C^\infty_D(\Omega))\subset\P_{1,D}^h$ and (cf. \cite[Section 4.4, p. 110]{brenner08})
\begin{align}\label{pih convergence1}
	&\mathcal I_h\phi\to\phi\quad\text{in } L^\infty(\Omega)
	&&\text{as }h\to 0\quad\forall \phi\in W^{1,p}(\Omega),\, p>N
	\\\label{pih convergence2}
	&\mathcal I_h\phi\to\phi\quad\text{in } H^1(\Omega)
	&&\text{as }h\to 0\quad\forall \phi\in H^2(\Omega).
\end{align}
Further, $\Psi_h$ denotes the $\P_{1,D}^h$-projection operator, i.e., $\Psi_h\colon H^1_D(\Omega)\to\P_{1,D}^h, p\mapsto y_h$ where $y_h$ solves
\begin{equation}\label{vi Psi}
	\int_\Omega\nabla y_h\cdot\nabla\phi_h + y_h\phi_h\d x = \int_\Omega\nabla p\cdot\nabla\phi_h + p\phi_h\d x\quad\forall \phi_h\in\P_{1,D}^h.
\end{equation}
Testing \cref{vi Psi} with $y_h - \phi_h$ for an arbitrarily fixed $\phi_h\in\P_{1,D}^h$, it holds that
\begin{align*}
	\|y_h - p\|_{H^1(\Omega)}^2 
 	&= (y_h - p,y_h - p)_{H^1(\Omega)} = (y_h - p, \phi_h - p)_{H^1(\Omega)}
	\\
	&\leq\|y_h - p\|_{H^1(\Omega)}\|\phi_h - p\|_{H^1(\Omega)}.
\end{align*}
Thus, for all $p\in H^1_D(\Omega)$, we have that
\begin{equation}\label{cea}
	\|\Psi_h p - p\|_{H^1(\Omega)}\leq \inf_{\phi_h\in\P_{1,D}^h}\|\phi_h - p\|_{H^1(\Omega)}.
\end{equation}
Furthermore, since $H^1_D(\Omega) = \overline{C^\infty_D(\Omega)}^{\|\cdot\|_{H^1(\Omega)}}$, for every $p\in H^1_D(\Omega)$ and for every $\epsilon>0$, there exists a $p_\epsilon\in C^\infty_D(\Omega)$ such that $\|p - p_\epsilon\|_{H^1(\Omega)}\leq\frac{\epsilon}{2}$, and due to \cref{pih convergence2}, there exists $ h_\epsilon>0$ such that $\|\mathcal I_hp_\epsilon - p_\epsilon\|_{H^1(\Omega)}\leq\frac{\epsilon}{2}$ for all $h\in (0,h_\epsilon]$. Then, along with \cref{cea} and since $\mathcal I_hp_\epsilon\in\P_{1,D}^h$, it follows for every $h\in (0,h_\epsilon]$ that
\begin{equation*}
	\|p - \Psi_h p\|_{H^1(\Omega)}\leq \|\mathcal I_h p_\epsilon - p\|_{H^1(\Omega)}\leq \|\mathcal I_h p_\epsilon - p_\epsilon\|_{H^1(\Omega)} + \|p_\epsilon - p\|_{H^1(\Omega)}\leq \epsilon.
\end{equation*}
This implies that 
\begin{equation}\label{convergence psi}
	\Psi_h p\to p\quad\text{in }H^1(\Omega)\quad\text{as }h\to 0\quad\forall p\in H^1_D(\Omega).
\end{equation}
Evaluating the first line in \cref{system state} at $t_{l+\frac{1}{2}}$ and the second line in \cref{system state} at $t_l$, and incorporating the continuous piecewise linear elements, we propose for every $\nu_h\in\mathcal V_{ad}$ the following fully discrete scheme:
\begin{equation}\label{leapfrog scheme}
	\left\{\begin{aligned}
		&\int_\Omega(\nu_h\delta p_h^{l+\frac{1}{2}} + \eta p_h^{l+\frac{1}{2}})\phi_h - \bm u_h^{l + \frac{1}{2}}\cdot\nabla\phi_h\d x = \int_\Omega F^{l+\frac{1}{2}}\phi_h\d x
		\\
		&\begin{aligned}
			&\phantom{\hspace{5cm}}
			&&\forall \phi_h\in\P_{1,D}^h, l=0,\dots, N-1
			\\
			&\delta\bm u_h^l + \nabla p_h^l = 0
			&&\forall l=1,\dots, N-1
			\\
			&p_h^0\coloneqq\Psi_hp_0,\quad \bm u_h^{\frac{1}{2}} \coloneqq\Phi_h(\nu_h).
		\end{aligned}
	\end{aligned}\right.
\end{equation}
Here, $\Phi_h\colon L^\infty(\Omega)\to \DGbm^h_0$ maps every $\nu\in L^\infty(\Omega)$ to $\nabla y_h$ where $y_h\in\P_{1,D}^h$ denotes the unique solution to 
\begin{equation}\label{definition Phi h}
	\int_\Omega\nabla y_h\cdot\nabla\phi_h\d x = \int_\Omega(\eta p_0 + \nu p_1)\phi_h\d x\quad\forall\phi_h\in\P_{1,D}^h.
\end{equation}
The system \cref{leapfrog scheme} can be solved by alternately solving the variational problem and computing the second line as follows: Suppose that $p_h^l\in\P_{1, D}^h$ and $\bm u_h^{l + \frac{1}{2}}\in\DGbm^h_0$ are already given for some $l=0,\dots, N-2$. Then, applying \cref{def deltaphk}, the finite-dimensional variational problem in \cref{leapfrog scheme} admits a unique solution $p_h^{l+1}\in\P_{1, D}^h$ due to the Lax-Milgram lemma. Furthermore, due to \cref{def deltauhk} and the second line in \cref{leapfrog scheme}, it follows that $\bm u_h^{l+\frac{3}{2}}= \bm u_h^{l+\frac{1}{2}} - \tau\nabla p_h^{l+1}$. Consequently, the iteration scheme \cref{leapfrog scheme} is well-defined.

\subsection{Stability Analysis}

Under a suitable {CFL}-condition, we prove a stability result for the discrete scheme \cref{leapfrog scheme}. Furthermore, we make use of the following inverse estimate:

\begin{lemma}[{\cite[Theorem 3.2.6]{ciarlet02}}]\label{lemma inverse estimate}
	There exists a constant $c_{inv}>0$, independent of $h$, such that
	\begin{equation*}
		\|\nabla p_h\|_{\bm L^2(\Omega)}\leq \frac{c_{inv}}{h}\|p_h\|_{L^2(\Omega)}\quad\forall p_h\in\P_{1}^h.
	\end{equation*}
\end{lemma}

\begin{theorem}\label{theorem stability}
	Let \cref{assumption} hold. Furthermore, let $h>0$ and $N\in\mathbb N$ satisfy the {CFL}-condition
	\begin{equation}\label{cfl}
		\frac{1}{Nh} = \frac{\tau}{h}\leq c_{cfl}\coloneqq \frac{\sqrt{\nu_-}}{\sqrt{2}c_{inv}}.
	\end{equation}
	Then, there exists a constant $C>0$, independent $h$ and $N$, such that for every $\nu_h\in\mathcal V_{ad}$, the associated unique solution $(\{p_h^l\}_{l=0}^N, \{\bm u_h^{l+\frac{1}{2}}\}_{l=0}^{N-1})\in(\P_{1,D}^h)^{N+1}\times(\DGbm_0^h)^N$ to the leapfrog scheme \cref{leapfrog scheme} satisfies
	\begin{align}\label{boundedness claim 1}
		\max_{l\in\{0,\dots,N-1\}}\|\delta p_h^{l+\frac{1}{2}}\|_{L^2(\Omega)} + \max_{l\in\{1,\dots,N-1\}}\|\delta\bm u_h^l\|_{\bm L^2(\Omega)} + \max_{l\in\{1,\dots,N-1\}}\|\nabla p_h^l\|_{\bm L^2(\Omega)} \leq C
		\\\label{boundedness claim 2}
		\max_{l\in\{0,\dots,N\}}\|p_h^l\|_{L^2(\Omega)} + \max_{l\in\{0,\dots,N-1\}}\|\bm u_h^{l+\frac{1}{2}}\|_{L^2(\Omega)}\leq C.
	\end{align}
\end{theorem}

\begin{proof}
	To begin with, we investigate the initial approximations. Testing the first line in \cref{leapfrog scheme} for $l=0$ with $\phi_h = p_h^{\frac{1}{2}}$ leads to 
	\begin{align*}
			&\int_\Omega(\nu_h\delta p_h^{\frac{1}{2}} + \eta p_h^{\frac{1}{2}} )p_h^{\frac{1}{2}} - \bm u_h^{\frac{1}{2}} \cdot\nabla p_h^{\frac{1}{2}} \d x = \int_\Omega F^{\frac{1}{2}} p_h^{\frac{1}{2}} \d x
			\\
			\underbrace{\Leftrightarrow}_{\cref{def deltaphk}}
			&\int_\Omega\frac{\nu_h}{2\tau}((p_h^1)^2 - (p_h^0)^2) + \eta (p_h^{\frac{1}{2}})^2 - \bm u_h^{\frac{1}{2}} \cdot\nabla p_h^{\frac{1}{2}} \d x = \int_\Omega F^{\frac{1}{2}} p_h^{\frac{1}{2}} \d x.
	\end{align*}
	Therefore, along with the inverse estimate from \cref{lemma inverse estimate}, it follows that
	\begin{align*}
		\frac{\nu_-}{\tau}\|p_h^{\frac{1}{2}} \|_{L^2(\Omega)}^2
		&\leq \int_\Omega\frac{\nu_h}{\tau}(p_h^\frac{1}{2})^2\d x = \int_\Omega\frac{\nu_h}{\tau}\left(\frac{p_h^1 + p_h^0}{2}\right)^2\d x \leq \int_\Omega\frac{\nu_h}{2\tau}((p_h^1)^2 + (p_h^0)^2)\d x
		\\
		&\leq\int_\Omega \frac{\nu_h}{\tau}(p_h^0)^2 + \bm u_h^{\frac{1}{2}} \cdot\nabla p_h^{\frac{1}{2}} \d x + F^{\frac{1}{2}} p_h^{\frac{1}{2}} \d x 
		\\
		&\leq \frac{\nu_+}{\tau}\|p_h^0\|_{L^2(\Omega)}^2 + \left(\frac{c_{inv}}{h}\|\bm u_h^{\frac{1}{2}} \|_{\bm L^2(\Omega)} + \|F^{\frac{1}{2}} \|_{L^2(\Omega)}\right) \|p_h^{\frac{1}{2}} \|_{L^2(\Omega)}.
	\end{align*}
	Then, applying Young's inequality, we obtain that 
	\begin{equation*}
		\frac{\nu_-}{\tau}\|p_h^{\frac{1}{2}} \|_{L^2(\Omega)}^2\leq \frac{\nu_+}{\tau}\|p_h^0\|_{L^2(\Omega)}^2 + \frac{\tau}{2\nu_-}\left(\frac{c_{inv}}{h}\|\bm u_h^{\frac{1}{2}} \|_{\bm L^2(\Omega)} \!+\! \|F^{\frac{1}{2}} \|_{L^2(\Omega)}\right)^2 \!+ \frac{\nu_-}{2\tau}\|p_h^{\frac{1}{2}} \|_{L^2(\Omega)}^2,
	\end{equation*}
	such that 
	\begin{align*}
		\frac{\nu_-}{2\tau}\|p_h^{\frac{1}{2}} \|_{L^2(\Omega)}^2
		&\leq \frac{\nu_+}{\tau}\|p_h^0\|_{L^2(\Omega)}^2 + \frac{\tau}{2\nu_-}\left(\frac{c_{inv}}{h}\|\bm u_h^{\frac{1}{2}} \|_{\bm L^2(\Omega)} + \|F^{\frac{1}{2}} \|_{L^2(\Omega)}\right)^2
		\\
		&\leq \frac{\nu_+}{\tau}\|p_h^0\|_{L^2(\Omega)}^2 + \frac{\tau c_{inv}^2}{\nu_-h^2}\|\bm u_h^{\frac{1}{2}} \|_{\bm L^2(\Omega)}^2 + \frac{\tau}{\nu_-}\|F^{\frac{1}{2}} \|_{L^2(\Omega)}^2.
	\end{align*}
	Multiplying both sides with $\frac{2\tau}{\nu_-}$ and making use of the CFL-condition \cref{cfl} yields that 
	\begin{align}\label{ieq ph12}
		\|p_h^{\frac{1}{2}} \|_{L^2(\Omega)}^2
		&\leq\frac{2\nu_+}{\nu_-}\|p_h^0\|_{L^2(\Omega)}^2 + \frac{2\tau^2 c_{inv}^2}{\nu_-^2h^2}\|\bm u_h^{\frac{1}{2}} \|_{\bm L^2(\Omega)}^2 + \frac{2\tau^2}{\nu_-^2}\|F^{\frac{1}{2}} \|_{L^2(\Omega)}^2
		\\\notag
		&\hspace{-.52cm}\underbrace{\leq}_{\cref{leapfrog scheme},\cref{cfl}}\frac{2\nu_+}{\nu_-}\|\Psi_hp_0\|_{L^2(\Omega)}^2 + \frac{1}{\nu_-}\|\Phi_h(\nu_h)\|_{\bm L^2(\Omega)}^2 + \frac{2T^2}{\nu_-^2}\|f\|_{L^1(I, L^2(\Omega))}^2,
	\end{align}
	where the inequality $\|F^{\frac{1}{2}}\|_{L^2(\Omega)}\leq\|f\|_{L^1(I, L^2(\Omega))}$ follows from $F^{\frac{1}{2}} = F(t_{\frac{1}{2}})= \int_0^{t_\frac{1}{2}}f(s)\d s$. Testing \cref{definition Phi h} with $\phi_h = \nabla y_h$ and utilizing the Poincaré inequality, we obtain the boundedness of $\{\Phi_h(\nu_h)\}_{h> 0}\subset\bm L^2(\Omega)$. Furthermore, $\{\Psi_hp_0\}_{h>0}\subset L^2(\Omega)$ is bounded due to \cref{convergence psi}. Altogether, concluding from \cref{ieq ph12}, we obtain the boundedness of $\{p_h^{\frac{1}{2}} \}_{h>0}\subset L^2(\Omega)$. Now, testing the first line in \cref{leapfrog scheme} for $l=0$ with $\phi_h = \delta p_h^{\frac{1}{2}}$, it holds that
	\begin{align*}
		\nu_-\|\delta p_h^{\frac{1}{2}} \|_{L^2(\Omega)}^2
		&\leq \int_\Omega\nu_h\delta (p_h^{\frac{1}{2}})^2\d x \underbrace{=}_{\cref{leapfrog scheme}} \int_\Omega (F^{\frac{1}{2}} - \eta p_h^{\frac{1}{2}} )\delta p_h^{\frac{1}{2}} + \Phi_h(\nu_h)\cdot\nabla \delta p_h^{\frac{1}{2}} \d x
		\\
		&\hspace{-.18cm}\underbrace{=}_{\cref{definition Phi h}}\int_\Omega (F^{\frac{1}{2}} - \eta p_h^{\frac{1}{2}} )\delta p_h^{\frac{1}{2}} + (\eta p_0 + \nu_h p_1)\delta p_h^{\frac{1}{2}} \d x
		\\
		&\leq(\|f\|_{L^1(I, L^2(\Omega))} + \|\eta\|_{L^\infty(\Omega)}(\|p_h^{\frac{1}{2}} \|_{L^2(\Omega)} + \|p_{0}\|_{L^2(\Omega)}) 
		\\
		&\qquad+ \nu_+\|p_1\|_{L^2(\Omega)})\|\delta p_h^{\frac{1}{2}} \|_{L^2(\Omega)},
	\end{align*}
	which leads to the boundedness of $\{\delta p_h^{\frac{1}{2}} \}_{h>0}\subset L^2(\Omega)$. Furthermore, it holds that
	\begin{align*}
		\|\delta\bm u_h^1\|_{\bm L^2(\Omega)}
		&\underbrace{=}_{\cref{leapfrog scheme}} \|\nabla p_h^1\|_{\bm L^2(\Omega)}\underbrace{=}_{\cref{def deltaphk}}\|\nabla(p_h^0 + \tau\delta p_h^{\frac{1}{2}} )\|_{\bm L^2(\Omega)}
		\\
		&\hspace{-.18cm}\underbrace{\leq}_{\text{Lem.}\ref{lemma inverse estimate}}\|\nabla\Psi_h(p_0)\|_{\bm L^2(\Omega)} + \frac{c_{inv}\tau}{h}\|\delta p_h^{\frac{1}{2}} \|_{L^2(\Omega)}
		\\
		&\hspace{-.05cm}\underbrace{\leq}_{\cref{cfl}}\|\nabla\Psi_h(p_0)\|_{\bm L^2(\Omega)} + c_{inv}c_{cfl}\|\delta p_h^{\frac{1}{2}} \|_{L^2(\Omega)}.
	\end{align*}
	Along with \cref{convergence psi}, the above inequality implies the boundedness of $\{\delta\bm u_h^1\}_{h>0}\subset\bm L^2(\Omega)$. Now, let $l\in\{1,\dots, N-1\}$ be arbitrarily fixed. Testing the first line in \cref{leapfrog scheme} for $l$ and $l-1$ with $\phi_h = \delta p_h^{l+\frac{1}{2}} + \delta p_h^{l-\frac{1}{2}}$ leads to
	\begin{align}\label{first ieq}
		&\int_\Omega(\nu_h\delta p_h^{l+\frac{1}{2}} + \eta p_h^{l+\frac{1}{2}})(\delta p_h^{l+\frac{1}{2}} + \delta p_h^{l-\frac{1}{2}}) - \bm u_h^{l + \frac{1}{2}}\cdot\nabla(\delta p_h^{l+\frac{1}{2}} + \delta p_h^{l-\frac{1}{2}})\d x
		\\\notag
		&=\int_\Omega F^{l+\frac{1}{2}}(\delta p_h^{l+\frac{1}{2}} + \delta p_h^{l-\frac{1}{2}})\d x
	\end{align}
	and 
	\begin{align}\label{second ieq}
		&\int_\Omega(\nu_h\delta p_h^{l-\frac{1}{2}} + \eta p_h^{l-\frac{1}{2}})(\delta p_h^{l+\frac{1}{2}} + \delta p_h^{l-\frac{1}{2}}) - \bm u_h^{l-\frac{1}{2}}\cdot\nabla(\delta p_h^{l+\frac{1}{2}} + \delta p_h^{l-\frac{1}{2}})\d x 
		\\\notag
		&=\int_\Omega F^{l-\frac{1}{2}}(\delta p_h^{l+\frac{1}{2}} + \delta p_h^{l-\frac{1}{2}})\d x.
	\end{align}
	Subtracting \cref{second ieq} from \cref{first ieq} provides that 
	\begin{align}\label{big sum}
		&\int_\Omega\nu_h(\delta p_h^{l+\frac{1}{2}} -
		\delta p_h^{l-\frac{1}{2}})(\delta p_h^{l+\frac{1}{2}} + \delta p_h^{l-\frac{1}{2}}) + \eta(p_h^{l+\frac{1}{2}} - p_h^{l-\frac{1}{2}})(\delta p_h^{l+\frac{1}{2}} + \delta p_h^{l-\frac{1}{2}})\d x
		\\\notag
		&= \int_\Omega(\bm u_h^{l + \frac{1}{2}} - \bm u_h^{l-\frac{1}{2}})\cdot\nabla(\delta p_h^{l+\frac{1}{2}} + \delta p_h^{l-\frac{1}{2}}) + (F^{l+\frac{1}{2}} - F^{l-\frac{1}{2}})(\delta p_h^{l+\frac{1}{2}} + \delta p_h^{l-\frac{1}{2}})\d x .
	\end{align}
	Now, for some arbitrary $N_0\in\{1,\dots, N-1\}$, we sum up the equation \cref{big sum} for $k =1,\dots, N_0$ and investigate the resulting terms seperately. For the first term, we obtain that
	\begin{align}\label{estimate 1}
		&\sum_{l=1}^{N_0}\int_\Omega\nu_h(\delta p_h^{l+\frac{1}{2}} -
		\delta p_h^{l-\frac{1}{2}})(\delta p_h^{l+\frac{1}{2}} + \delta p_h^{l-\frac{1}{2}})\d x
		\\
		&\notag= \|\sqrt{\nu_h}\delta p_h^{N_0+\frac{1}{2}}\|_{L^2(\Omega)}^2 - \|\sqrt{\nu_h}\delta p_h^{\frac{1}{2}} \|_{L^2(\Omega)}^2\geq \nu_-\|\delta p_h^{N_0+\frac{1}{2}}\|_{L^2(\Omega)}^2 - \nu_+\|\delta p_h^{\frac{1}{2}} \|_{L^2(\Omega)}^2.
	\end{align}
	Furthermore, it holds that 
	\begin{equation}\label{estimate 2}
		\sum_{l=1}^{N_0}\int_\Omega\eta(p_h^{l+\frac{1}{2}} - p_h^{l-\frac{1}{2}})(\delta p_h^{l+\frac{1}{2}} + \delta p_h^{l-\frac{1}{2}})\d x\underbrace{=}_{\cref{def deltaphk}}\frac{\tau}{2}\sum_{l=1}^{N_0}\int_\Omega\eta(\delta p_h^{l+\frac{1}{2}} + \delta p_h^{l-\frac{1}{2}})^2\d x\geq 0.
	\end{equation}
	Moreover, we have that
	\begin{align}\label{long estimate}
		&\sum_{l=1}^{N_0}\int_\Omega(\bm u_h^{l + \frac{1}{2}} - \bm u_h^{l-\frac{1}{2}})\cdot\nabla(\delta p_h^{l+\frac{1}{2}} + \delta p_h^{l-\frac{1}{2}})\d x
		\\\notag
 		&\hspace{-.15cm}\underbrace{=}_{\cref{def deltauhk}}\tau\sum_{l=1}^{N_0}\int_\Omega\delta\bm u_h^l\cdot\nabla(\delta p_h^{l+\frac{1}{2}} + \delta p_h^{l-\frac{1}{2}})\d x 
		\\\notag
		&\hspace{-.15cm}\underbrace{=}_{\cref{def deltaphk}} \sum_{l=1}^{N_0}\int_\Omega\delta\bm u_h^l\cdot\nabla(p_h^{l+1} - p_h^l) + \delta\bm u_h^l\cdot\nabla(p_h^l - p_h^{l-1})\d x
		\\\notag
		&\hspace{-.15cm}\underbrace{=}_{\cref{leapfrog scheme}} \int_\Omega-\sum_{l=1}^{N_0-1}\delta\bm u_h^l\cdot(\delta\bm u_h^{l+1} - \delta\bm u_h^l) - \sum_{l=2}^{N_0} \delta\bm u_h^l\cdot(\delta\bm u_h^l - \delta\bm u_h^{l-1})\d x
		\\\notag
		&\quad + \int_\Omega \delta\bm u_h^{N_0}\cdot\nabla(p_h^{N_0 + 1} - p_h^{N_0})+\delta\bm u_h^1\cdot\nabla(p_h^1 - p_h^0)\d x
		\\\notag
		&= \int_\Omega-\sum_{l=1}^{N_0-1}\delta\bm u_h^l\cdot(\delta\bm u_h^{l+1} - \delta\bm u_h^l) - \sum_{l=1}^{N_0-1} \delta\bm u_h^{l+1}\cdot(\delta\bm u_h^{l+1} - \delta\bm u_h^l)\d x
		\\\notag
		&\quad + \tau\int_\Omega\delta\bm u_h^{N_0}\cdot\nabla\delta p_h^{N_0 + \frac{1}{2}} + \delta\bm u_h^1\nabla\delta p_h^{\frac{1}{2}} \d x
		\\\notag
		&= \int_\Omega-\sum_{l=1}^{N_0-1}(\delta\bm u_h^{l+1} + \delta\bm u_h^l)\cdot(\delta\bm u_h^{l+1} - \delta\bm u_h^l)\d x
		\\\notag
		&\quad+ \tau\int_\Omega\delta\bm u_h^{N_0}\cdot\nabla\delta p_h^{N_0 + \frac{1}{2}} + \delta\bm u_h^1\nabla\delta p_h^{\frac{1}{2}} \d x
		\\\notag
		& = -\|\delta\bm u_h^{N_0}\|_{\bm L^2(\Omega)}^2 +\|\delta\bm u_h^1\|_{\bm L^2(\Omega)}^2 + \tau\int_\Omega\delta\bm u_h^{N_0}\cdot\nabla\delta p_h^{N_0 + \frac{1}{2}} + \delta\bm u_h^1\nabla\delta p_h^{\frac{1}{2}} \d x
		\\\notag
		&\leq -\|\delta\bm u_h^{N_0}\|_{\bm L^2(\Omega)}^2 +\|\delta\bm u_h^1\|_{\bm L^2(\Omega)}^2 + \frac{1}{2}\|\delta\bm u_h^{N_0}\|_{\bm L^2(\Omega)}^2 + \frac{\tau^2}{2}\|\nabla\delta p_h^{N_0 + \frac{1}{2}}\|_{\bm L^2(\Omega)}^2
		\\\notag
		&\quad+ \frac{1}{2}\|\delta\bm u_h^1\|_{\bm L^2(\Omega)}^2 + \frac{\tau^2}{2}\|\nabla\delta p_h^{\frac{1}{2}} \|_{\bm L^2(\Omega)}^2
		\\\notag
		&\hspace{-.31cm}\underbrace{\leq}_{\text{Lem.} \ref{lemma inverse estimate}} - \frac{1}{2}\|\delta\bm u_h^{N_0}\|_{\bm L^2(\Omega)}^2 + \frac{3}{2}\|\delta\bm u_h^1\|_{\bm L^2(\Omega)}^2 + \frac{c_{inv}^2\tau^2}{2h^2}\|\delta p_h^{N_0 + \frac{1}{2}}\|_{L^2(\Omega)}^2 + \frac{c_{inv}^2\tau^2}{2h^2}\|\delta p_h^{\frac{1}{2}} \|_{L^2(\Omega)}^2
		\\\notag
		&\hspace{-.18cm}\underbrace{\leq}_{\cref{cfl}} - \frac{1}{2}\|\delta\bm u_h^{N_0}\|_{\bm L^2(\Omega)}^2 + \frac{3}{4}\|\delta\bm u_h^1\|_{\bm L^2(\Omega)}^2 + \frac{\nu_-}{4}\|\delta p_h^{N_0 + \frac{1}{2}}\|_{L^2(\Omega)}^2 + \frac{\nu_-}{4}\|\delta p_h^{\frac{1}{2}} \|_{L^2(\Omega)}^2.
	\end{align}
	Taking the last term from \cref{big sum} into account, we observe that 
	\begin{align}\label{ieq F1}
		&\sum_{l=1}^{N_0}\int_\Omega(F^{l+\frac{1}{2}} - F^{l-\frac{1}{2}})(\delta p_h^{l+\frac{1}{2}} + \delta p_h^{l-\frac{1}{2}})\d x 
		\\\notag
		&\leq \sum_{l=1}^{N_0}\|F^{l+\frac{1}{2}} - F^{l-\frac{1}{2}}\|_{L^2(\Omega)}(\|\delta p_h^{l+\frac{1}{2}}\|_{L^2(\Omega)} + \|\delta p_h^{l-\frac{1}{2}}\|_{L^2(\Omega)})
		\\\notag
		&=\sum_{l=1}^{N_0}\|F^{l+\frac{1}{2}} - F^{l-\frac{1}{2}}\|_{L^2(\Omega)}\|\delta p_h^{l+\frac{1}{2}}\|_{L^2(\Omega)} + \sum_{l=0}^{N_0-1}\|F^{l+ \frac{3}{2}} - F^{l+\frac{1}{2}}\|_{L^2(\Omega)}\|\delta p_h^{l+\frac{1}{2}}\|_{L^2(\Omega)}
		\\\notag
		&=\sum_{l=1}^{N_0-1}(\|F^{l+\frac{1}{2}} - F^{l-\frac{1}{2}}\|_{L^2(\Omega)} + \|F^{l+ \frac{3}{2}} - F^{l+\frac{1}{2}}\|_{L^2(\Omega)})\|\delta p_h^{l+\frac{1}{2}}\|_{L^2(\Omega)}
		\\\notag
		&\quad + \|F^{N_0+\frac{1}{2}} - F^{N_0-\frac{1}{2}}\|_{L^2(\Omega)}\|\delta p_h^{N_0+\frac{1}{2}}\|_{L^2(\Omega)}+ \|F^{\frac{3}{2}} - F^{\frac{1}{2}} \|_{L^2(\Omega)}\|\delta p_h^{\frac{1}{2}} \|_{L^2(\Omega)}.
	\end{align}
	Since $F(t)= \int_0^{t}f(s)\d s$ for every $t\in I$, it holds for all $l\in \{1,\dots,N-1\}$ that 
	\begin{equation}\label{ieq F2}
		\|F^{l+\frac{1}{2}} - F^{l-\frac{1}{2}}\|_{L^2(\Omega)} = \left\|\int_{t_{l-\frac{1}{2}}}^{t_{l+\frac{1}{2}}}f(s)\d s\right\|\leq \|f\|_{L^1(I, L^2(\Omega))}.
	\end{equation}
	Applying \cref{ieq F2} to \cref{ieq F1} along with Young's inequality, it follows that
	\begin{align}\label{ieq F3}
		&\sum_{l=1}^{N_0}\int_\Omega(F^{l+\frac{1}{2}} - F^{l-\frac{1}{2}})(\delta p_h^{l+\frac{1}{2}} + \delta p_h^{l-\frac{1}{2}})\d x 
		\\\notag
		&\leq\sum_{l=1}^{N_0-1}(\|F^{l+\frac{1}{2}} - F^{l-\frac{1}{2}}\|_{L^2(\Omega)} + \|F^{l+ \frac{3}{2}} - F^{l+\frac{1}{2}}\|_{L^2(\Omega)})\|\delta p_h^{l+\frac{1}{2}}\|_{L^2(\Omega)}
		\\\notag
		&\quad + \frac{2}{\nu_-}\|f\|_{L^1(I, L^2(\Omega))}^2 + \frac{\nu_-}{4}\|\delta p_h^{N_0+\frac{1}{2}}\|_{L^2(\Omega)}^2 + \frac{\nu_-}{4}\|\delta p_h^{\frac{1}{2}} \|_{L^2(\Omega)}^2.
	\end{align}
	Now, applying \cref{estimate 1}-\cref{long estimate}, and \cref{ieq F3} to \cref{big sum}, we obtain that 
	\begin{align*}
		&\nu_-\|\delta p_h^{N_0+\frac{1}{2}}\|_{L^2(\Omega)}^2 - \nu_+\|\delta p_h^{\frac{1}{2}} \|_{L^2(\Omega)}^2
		\\
		&\leq - \frac{1}{2}\|\delta\bm u_h^{N_0}\|_{\bm L^2(\Omega)}^2 + \frac{3}{2}\|\delta\bm u_h^1\|_{\bm L^2(\Omega)}^2 + \frac{\nu_-}{2}\|\delta p_h^{N_0 + \frac{1}{2}}\|_{L^2(\Omega)}^2 + \frac{\nu_-}{2}\|\delta p_h^{\frac{1}{2}} \|_{L^2(\Omega)}^2
		\\
		&\quad+ \sum_{l=1}^{N_0-1}(\|F^{l+\frac{1}{2}} - F^{l-\frac{1}{2}}\|_{L^2(\Omega)} + \|F^{l+ \frac{3}{2}} - F^{l+\frac{1}{2}}\|_{L^2(\Omega)})\|\delta p_h^{l+\frac{1}{2}}\|_{L^2(\Omega)} 
		\\
		&\quad+ \frac{2}{\nu_-}\|f\|_{L^1(I,L^2(\Omega))}^2.
	\end{align*}
	Rearranging yields that 
	\begin{align}\label{ieq stability}
		&\frac{\nu_-}{2}\|\delta p_h^{N_0+\frac{1}{2}}\|_{L^2(\Omega)}^2 + \frac{1}{2}\|\delta\bm u_h^{N_0}\|_{\bm L^2(\Omega)}^2
		\\\notag
		&\leq\left(\nu_+ + \frac{\nu_-}{2}\right)\|\delta p_h^{\frac{1}{2}} \|_{L^2(\Omega)}^2 + \frac{3}{2}\|\delta\bm u_h^1\|_{\bm L^2(\Omega)}^2 + \frac{2}{\nu_-}\|f\|_{L^1(I,L^2(\Omega))}^2.
		\\\notag
		&\quad+ \sum_{l=1}^{N_0-1}(\|F^{l+\frac{1}{2}} - F^{l-\frac{1}{2}}\|_{L^2(\Omega)} + \|F^{l+ \frac{3}{2}} - F^{l+\frac{1}{2}}\|_{L^2(\Omega)})\|\delta p_h^{l+\frac{1}{2}}\|_{L^2(\Omega)} 
	\end{align}
	Defining 
	\begin{align*}
		\gamma
		&\coloneqq 2\min\{\nu_-, 1\}^{-1},
		\\
		\alpha
		&\coloneqq \left(\nu_+ + \frac{\nu_-}{2}\right)\|\delta p_h^{\frac{1}{2}} \|_{L^2(\Omega)}^2 + \frac{3}{2}\|\delta\bm u_h^1\|_{\bm L^2(\Omega)}^2 + \frac{2}{\nu_-}\|f\|_{L^1(I,L^2(\Omega))}^2,
		\\
		h_l
		&\coloneqq \|F^{l+\frac{1}{2}} - F^{l-\frac{1}{2}}\|_{L^2(\Omega)} + \|F^{l+ \frac{3}{2}} - F^{l+\frac{1}{2}}\|_{L^2(\Omega)}\quad\forall l\in\{1, \dots, N-2\},
		\\
		I_0
		&\coloneqq \{l\in\{1,\dots,N_0-1\}:\|\delta p_h^{l+\frac{1}{2}}\|_{L^2(\Omega)}\geq 1\},\quad J_0 \coloneqq \{1,\dots, N_0-1\}\setminus I_0,
	\end{align*}
	the inequality \cref{ieq stability} is equivalent to
	\begin{align}\label{ieq delta phN0}
		&\|\delta p_h^{N_0+\frac{1}{2}}\|_{L^2(\Omega)}^2 + \|\delta\bm u_h^{N_0}\|_{\bm L^2(\Omega)}^2
 		\\\notag
		&\leq \gamma\alpha + \gamma\sum_{l\in J_0}h_l\|\delta p_h^{l+\frac{1}{2}}\|_{L^2(\Omega)} + \gamma\sum_{l\in I_0}h_l\|\delta p_h^{l+\frac{1}{2}}\|_{L^2(\Omega)}
		\\\notag
		&\leq \gamma\alpha + \gamma\sum_{l\in J_0}h_l + \gamma\sum_{l\in I_0}h_l(\|\delta p_h^{l+\frac{1}{2}}\|_{L^2(\Omega)}^2 + \|\delta\bm u_h^l\|_{\bm L^2(\Omega)}^2).
	\end{align}
	By definition, it holds that 
	\begin{align}\label{ieq hl}
		\sum_{l=1}^{N_0-1}h_l 
		&= \sum_{l=1}^{N_0-1}(\|F^{l+\frac{1}{2}} - F^{l-\frac{1}{2}}\|_{L^2(\Omega)} + \|F^{l+ \frac{3}{2}} - F^{l+\frac{1}{2}}\|_{L^2(\Omega)})
		\\\notag
		&= \sum_{l=1}^{N_0-1}\left(\left\|\int_{t_l-\tau/2}^{t_l+\tau/2}f(s)\d s\right\|_{L^2(\Omega)} + \left\|\int_{t_{l+1}-\tau/2}^{t_{l+1}+\tau/2}f(s)\d s\right\|_{L^2(\Omega)}\right)
		\\\notag
		&\leq \sum_{l=1}^{N_0-1}\int_{t_l-\tau/2}^{t_l+\tau/2}\|f(s)\|_{L^2(\Omega)}\d s + \sum_{l=1}^{N_0-1}\int_{t_{l+1}-\tau/2}^{t_{l+1}+\tau/2}\|f(s)\|_{L^2(\Omega)}\d s
		\\\notag
		&\leq 2\|f\|_{L^1(I, L^2(\Omega))}.
	\end{align}
	Applying \cref{ieq hl} to \cref{ieq delta phN0}, it follows that 
	\begin{align*}
		&\|\delta p_h^{N_0+\frac{1}{2}}\|_{L^2(\Omega)}^2 + \|\delta\bm u_h^{N_0}\|_{\bm L^2(\Omega)}^2
		\\
		&\leq \gamma(\alpha + 2\|f\|_{L^1(I,L^2(\Omega))}) + \gamma\sum_{l\in I_0}h_l(\|\delta p_h^{l+\frac{1}{2}}\|_{L^2(\Omega)}^2 + \|\delta\bm u_h^l\|_{\bm L^2(\Omega)}^2).
	\end{align*}
	Now, the discrete Gronwall lemma (cf. \cite{clark87}) implies that 
	\begin{align*}
		\|\delta p_h^{N_0+\frac{1}{2}}\|_{L^2(\Omega)}^2+ \|\delta\bm u_h^{N_0}\|_{\bm L^2(\Omega)}^2
		&\leq \gamma(\alpha + 2\|f\|_{L^1(I,L^2(\Omega))})\exp\left(\gamma \sum_{l=1}^{N_0-1}h_l\right)
		\\
		&\hspace{-.18cm}\underbrace{\leq}_{\cref{ieq hl}}\gamma(\alpha + 2\|f\|_{L^1(I,L^2(\Omega))})\exp\left(2\gamma \|f\|_{L^1(I,L^2(\Omega))}\right).
	\end{align*}
	Since this inequality holds for all $N_0\in\{1,\dots, N-1\}$, along with the boundedness of $\{\delta p_h^{\frac{1}{2}} \}_{h>0}\subset L^2(\Omega)$ and \cref{leapfrog scheme}, we conclude that the claim \cref{boundedness claim 1} is valid. Furthermore, by \cref{def deltaphk} along with the triangular inequality, we obtain that
	\begin{equation*}
		\|p_h^{l+1}\|_{L^2(\Omega)}\leq \|p_h^l\|_{L^2(\Omega)} + \tau\|\delta p_h^{l+\frac{1}{2}}\|_{L^2(\Omega)} \leq \|p_h^l\|_{L^2(\Omega)} + \tau C\quad\forall l\in \{0,\dots, N-1\}.
	\end{equation*}
	By induction and since $\tau = \frac{1}{N}$, this leads to
	\begin{equation*}
		\|p_h^{l+1}\|_{L^2(\Omega)}\leq (l+1)\tau C + \|p_h^0\|_{L^2(\Omega)}\leq C + \|p_h^0\|_{L^2(\Omega)}\quad\forall l\in \{0,\dots, N-1\}.
	\end{equation*}
	Analogously, we conclude
	\begin{equation*}
		\|\bm u_h^{l+\frac{1}{2}}\|_{\bm L^2(\Omega)}\leq C + \|\bm u_h^{\frac{1}{2}} \|_{\bm L^2(\Omega)}\quad\forall l\in \{0,\dots, N-1\}.
	\end{equation*}
	Thus, the second claim \cref{boundedness claim 2} is also valid.
\end{proof}

\section{\texorpdfstring{Fully Discrete Scheme for \cref{P}}{}}
\label{section discrete optimal control problem}

Let us begin by introducing the discrete admissible set associated with $\mathcal V_{ad}$ as follows:
\begin{equation*}
	\mathcal V_{ad}^h \coloneqq\{\nu\in \DG_0^h: \nu_- \leq \nu(x)\leq \nu_+ \text{ for every }x\in\Omega\}\subset\mathcal V_{ad}.
\end{equation*}
Making use of the midpoint rule for the numerical integration in time and the leapfrog scheme \cref{leapfrog scheme}, we propose the following fully discrete scheme for \cref{P}:
\begin{equation}\label{P discretized}
	\left\{\begin{aligned}
		&\min\mathcal J_h(\nu_h, \{p_h^l\}_{l=0}^{N}) \coloneqq\frac{\tau}{2} \sum_{i=1}^m\sum_{l=0}^{N-1}\int_{\Omega}\!a_i(t_{l+\frac{1}{2}})(p_h^{l+\frac{1}{2}}-p^{ob}_{i}(t_{l+\frac{1}{2}}))^2\d x\!+\!\frac{\lambda}{2}\|\nu_h\|_{L^2(\Omega)}^2\\
		&\text{ s.t. $\nu_h\in\mathcal V_{ad}^h$ and $(\{p_h^l\}_{l=0}^N,\{\bm u_h^{l+\frac{1}{2}}\}_{l=0}^{N-1})\in(\P_{1,D}^h)^{N+1}\times(\DGbm_0^h)^N$ solves \cref{leapfrog scheme}}.
	\end{aligned}\right.
\end{equation}
As for every $\nu_h\in\mathcal V_{ad}^h$, \cref{leapfrog scheme} admits a unique solution $(\{p_h^l\}_{l=0}^N,\{\bm u_h^{l+\frac{1}{2}}\}_{l=0}^{N-1})\in(\P_{1,D}^h)^{N+1}\times(\DGbm_0^h)^N$, the associated solution mapping 
\begin{equation*}
	S_h\colon \mathcal V_{ad}^h\to L^2(\Omega)^{N+1}\times\bm L^2(\Omega)^N, \nu_h\mapsto(\{p_h^l\}_{l=0}^N,\{\bm u_h^{l+\frac{1}{2}}\}_{l=0}^{N-1})	
\end{equation*}
is well-defined. Denoting the first component of $S_h$ with $S_{h,p}\colon\nu_h\mapsto\{p_h^l\}_{l=0}^N$, the reduced formulation of \cref{P discretized} reads as follows:
\begin{equation}\label{P discretized reduced}\tag{$\P_h$}
	\min_{\nu\in\mathcal V_{ad}^h} J_h(\nu_h)\coloneqq \mathcal J_h(\nu_h, S_{h,p}(\nu_h)).
\end{equation}

\begin{theorem}\label{theorem existence}
	Let \cref{assumption} hold. Then, for every $h>0$, the minimization problem \cref{P discretized reduced} admits at least one global minimizer $\nu_h\in \mathcal V_{ad}^h$. 
\end{theorem}

\begin{proof}
	Let $h>0$ be given. We show that the solution operator $S_{h,p}\colon L^2(\Omega)\supset\mathcal V_{ad}^h\to L^2(\Omega)^{N+1},\nu_h\mapsto \{p_h^l\}_{l=0}^N$ is continuous. For this purpose, let $\nu_h,\tilde\nu_h\in\mathcal V_{ad}^h$ and $(\{p_h^l\}_{l=0}^N,\{\bm u_h^{l+\frac{1}{2}}\}_{l=0}^{N-1})$, $(\{\tilde p_h^l\}_{l=0}^N,\{\tilde{\bm u}_h^{l+\frac{1}{2}}\}_{l=0}^{N-1})\in(\P_{1,D}^h)^{N+1}\times(\DGbm_0^h)^N$ denote the solutions to \cref{leapfrog scheme} associated with $\nu_h$ and $\tilde\nu_h$, respectively. Then, we have that 
	\begin{equation}\label{leapfrog scheme difference}
		\left\{\begin{aligned}
			&\int_\Omega(\tilde\nu_h(\delta\tilde p_h^{l+\frac{1}{2}} - \delta p_h^{l+\frac{1}{2}}) + \eta(\tilde p_h^{l+\frac{1}{2}} - p_h^{l+\frac{1}{2}}))\phi_h - (\tilde{\bm u}_h^{l + \frac{1}{2}} - \bm u_h^{l + \frac{1}{2}})\cdot\nabla\phi_h\d x 
			\\
			&\begin{aligned}
				&= \int_\Omega (\tilde\nu_h - \nu_h)\delta p_h^{l+\frac{1}{2}}\phi_h\d x
				&&\forall \phi_h\in\P_{1,D}^h, l=0,\dots, N-1
				\\
				&\delta\tilde{\bm u}_h^l - \delta \bm u_h^l + \nabla(\tilde p_h^l - p_h^l) = 0
				&&\forall l=1,\dots, N-1
				\\
				&\tilde p_h^0 - p_h^0 = 0,\quad \tilde{\bm u}_h^{\frac{1}{2}} - \bm u_h^{\frac{1}{2}} = \Phi_h(\tilde\nu_h) - \Phi_h(\nu_h).
			\end{aligned}
		\end{aligned}\right.
	\end{equation}
	We show by induction that for every $l=1,\dots, N$, there exists a constant $c>0$ that is independent of $\tilde\nu_h$ such that
	\begin{equation}\label{induction continuity}
		\left\{\begin{aligned}
			\|\tilde p_h^l - p_h^l\|_{L^2(\Omega)}
			&\leq c\|\tilde\nu_h - \nu_h\|_{L^2(\Omega)}
			\\
			\|\tilde{\bm u}_h^{l-\frac{1}{2}} - \bm u_h^{l-\frac{1}{2}}\|_{\bm L^2(\Omega)}
			&\leq c\|\tilde\nu_h - \nu_h\|_{L^2(\Omega)}.
		\end{aligned}\right.
	\end{equation}
	From \cref{leapfrog scheme difference}, we obtain for all $l=0,\dots, N-1$ that for every $\phi_h\in\P_{1,D}^h$ that
	\begin{align}\label{first equation continuity}
		&\int_\Omega\left(\frac{\tilde\nu_h}{\tau} + \frac{\eta}{2}\right)(\tilde p_h^{l+1} - p_h^{l+1})\phi_h\d x
		\\\notag
		&= \int_\Omega\left(\frac{\tilde\nu_h}{\tau} - \frac{\eta}{2}\right)(\tilde p_h^l - p_h^l)\phi_h + (\tilde{\bm u}_h^{l+\frac{1}{2}} - \bm u_h^{l+\frac{1}{2}})\cdot\nabla\phi_h + (\tilde\nu_h - \nu_h)\delta p_h^{l+\frac{1}{2}}\phi_h\d x.
	\end{align}
	The equation \cref{first equation continuity} with $l=0$ implies that
	\begin{align}\label{first equation continuity 1}
		\int_\Omega\left(\frac{\tilde\nu_h}{\tau} + \frac{\eta}{2}\right)(\tilde p_h^1 - p_h^1)\phi_h\d x
		&\underbrace{=}_{\cref{leapfrog scheme difference}} \int_\Omega(\Phi_h(\tilde\nu_h) - \Phi_h(\nu_h))\cdot\nabla\phi_h + (\tilde\nu_h - \nu_h)\delta p_h^{\frac{1}{2}}\phi_h\d x
		\\\notag
		&\hspace{-.04cm}\underbrace{=}_{\cref{definition Phi h}} \int_\Omega (\tilde\nu_h - \nu_h)(p_1 + \delta p_h^{\frac{1}{2}})\phi_h\d x\quad\forall\phi_h\in\P_{1,D}^h.
	\end{align}
	Thus, testing \cref{first equation continuity 1} with $\phi_h = (\tilde p_h^1 - p_h^1)$ and using H\"older's inequality, we obtain that 
	\begin{equation*}
		\|\tilde p_h^1 - p_h^1\|_{L^2(\Omega)}\leq \nu_{\min^{-1}}\tau(\|p_1\|_{L^\infty(\Omega)} + \|\delta p_h^{\frac{1}{2}}\|_{C(\overline\Omega)})\|\tilde\nu_h - \nu_h\|_{L^2(\Omega)}.
	\end{equation*}
	Furthermore, we have that 
	\begin{equation*}
		\|\tilde{\bm u}_h^{\frac{1}{2}} - \bm u_h^{\frac{1}{2}}\|_{\bm L^2(\Omega)} \underbrace{=}_{\cref{leapfrog scheme difference}}\|\Phi_h(\tilde\nu_h) - \Phi(\nu_h)\|_{\bm L^2(\Omega)}\leq c_P\|p_1\|_{L^\infty(\Omega)}\|\tilde\nu_h - \nu_h\|_{L^2(\Omega)},
	\end{equation*}
	where the last inequality is obtained by inserting $\phi_h = y_h$ in \cref{definition Phi h} and $c_P>0$ denotes a Poincaré constant. Therefore, \cref{induction continuity} is valid for $l=1$. Now, assume that \cref{induction continuity} holds for a fixed $l\in\{1,\dots,N-1\}$. Since
	\begin{equation}\label{convergence u}
		\|\tilde{\bm u}_{h}^{l+\frac{1}{2}} - \bm u_{h}^{l+\frac{1}{2}}\|_{\bm L^2(\Omega)}\underbrace{\leq}_{\cref{leapfrog scheme difference}}\|\tilde{\bm u}_{h}^{l-\frac{1}{2}} - \bm u_{h}^{l-\frac{1}{2}}\|_{\bm L^2(\Omega)} + \tau\|\nabla(\tilde p_h^l - p_h^l)\|_{\bm L^2(\Omega)},
	\end{equation}
	applying the induction assumption and the inverse estimate from \cref{lemma inverse estimate} to \cref{convergence u}, we obtain the second inequality in \cref{induction continuity} for $l+1$. Further, testing \cref{first equation continuity} with $\phi_h = \tilde p_h^{l+1} - p_h^{l+1}$ and again using the inverse estimate in \cref{lemma inverse estimate} yields that
	\begin{align*}
		&\|\tilde p_h^{l+1} - p_h^{l+1}\|_{L^2(\Omega)}
		\\
		&\leq \nu_-^{-1}\tau\bigg(\left(\frac{\nu_+}{\tau} + \frac{1}{2}\|\eta\|_{L^\infty(\Omega)}\right)\|\tilde p_h^l - p_h^l\|_{L^2(\Omega)}+\frac{c_{inv}}{h}\|\tilde{\bm u}_h^{l+\frac{1}{2}} - \bm u_h^{l+\frac{1}{2}}\|_{\bm L^2(\Omega)}
		\\
		&\quad+\|\delta p_h^{l+\frac{1}{2}}\|_{C(\overline\Omega)}\|\tilde\nu_h - \nu_h\|_{L^2(\Omega)}\bigg).
	\end{align*}
	Applying the induction assumption and \cref{convergence u}, we obtain the first inequality in \cref{induction continuity} for $l+1$. This completes the induction proof. Consequently, by \cref{induction continuity} and since $\tilde p_h^0 = p_h^0$, the mapping $S_{h,p}\colon L^2(\Omega)\supset\mathcal V_{ad}^h\to L^2(\Omega)^{N+1},\nu\mapsto\{p_h^l\}_{l=0}^N$ is continuous and thus $J_h\colon L^2(\Omega)\supset\mathcal V_{ad}^h\to\mathbb R$ is continuous. On the other hand, $\mathcal V_{ad}^h\subset\DG_0^h$ is compact since it is a closed and bounded subset of the finite-dimensional space $\DG_0^h$. Thus, the minimization problem \cref{P discretized reduced} admits at least one global minimizer due to the Weierstrass theorem. 
\end{proof}

\section{Convergence Analysis}
\label{section convergence}

Given $h>0$, $N\in\mathbb N$, and $(\{p_h^l\}_{l=0}^N, \{\bm u_h^{l+\frac{1}{2}}\}_{l=0}^{N-1})\in(\P_{1,D}^h)^{N+1}\times(\DGbm_0^h)^N$, let us define the interpolations
\begin{align*}
	\Lambda^{p}_{N,h}, \Pi^{p}_{N,h}, \Theta^{p}_{N,h}&\colon I\to \P_{1,D}^h
	&&\Lambda^{\bm u}_{N,h}, \Pi^{\bm u}_{N,h}\colon I\to \DGbm_0^h
	\\
	F_{N,h}, p_{i,N,h}^{ob}&\colon I\to L^2(\Omega)
	&&\hspace{.9cm}a_{i,N,h}\colon I\to L^\infty(\Omega),
\end{align*}
by 
\begin{align}\label{def lambda p}
	\Lambda^{p}_{N,h}(t)
	&\coloneqq \left\{\begin{aligned}
		&p_h^0
		&&\text{if } t = 0
		\\
		&p_h^l + (t - t_l)\delta p_h^{l+\frac{1}{2}}\hspace{.15cm}
		&&\text{if } t\in(t_l,t_{l+1}]\text{ for some }l\in \{0,\dots, N-1\}
	\end{aligned}\right.\hspace{-.21cm}
	\\\label{def pi p}
	\Pi^{p}_{N,h}(t) 
	&\coloneqq \left\{\begin{aligned}
		&p_h^0
		&&\text{if } t = 0
		\\
		&p_h^{l+\frac{1}{2}}\hspace{2.17cm}
		&&\text{if } t\in(t_l,t_{l+1}]\text{ for some }l\in \{0,\dots, N-1\}
	\end{aligned}\right.\hspace{-.21cm}
	\\\label{def theta p}
	\Theta^{p}_{N,h}(t) 
	&\coloneqq \left\{\begin{aligned}
		&p_h^0
		&&\text{if } t = 0
		\\
		&p_h^l\hspace{2.52cm}
		&&\text{if } t\in(t_l,t_{l+1}]\text{ for some }l\in \{0,\dots, N-1\}
	\end{aligned}\right.\hspace{-.21cm}
	\\\label{def lambda u}
	\Lambda^{\bm u}_{N,h}(t)
	&\coloneqq \left\{\begin{aligned}
		&\bm u_h^{\frac{1}{2}} 
		&&\text{if } t = 0
		\\
		&\bm u_h^{l-\frac{1}{2}} + (t - t_l)\delta\bm u_h^l
		&&\text{if } t\in(t_l,t_{l+1}]\text{ for some }l\in \{0,\dots, N-1\}
	\end{aligned}\right.\hspace{-.21cm}
	\\\label{def pi u}
	\Pi^{\bm u}_{N,h}(t)
	&\coloneqq \left\{\begin{aligned}
		&\bm u_h^{\frac{1}{2}} 
		&&\text{if } t = 0
		\\
		&\bm u_h^{l+\frac{1}{2}}\hspace{2.09cm}
		&&\text{if } t\in(t_l,t_{l+1}]\text{ for some }l\in \{0,\dots, N-1\}
	\end{aligned}\right.\hspace{-.21cm}
	\\\label{def F interpolate}
	F_{N,h}(t)
	&\coloneqq \left\{\begin{aligned}
		&F(0)
		&&\text{if } t = 0
		\\
		& F(t_{l+\frac{1}{2}})\hspace{1.68cm}
		&&\text{if } t\in(t_l,t_{l+1}]\text{ for some }l\in \{0,\dots, N-1\}
	\end{aligned}\right.\hspace{-.21cm}
	\\\label{def pob interpolate}
		p_{i,N,h}^{ob}(t)
	&\coloneqq \left\{\begin{aligned}
		&p_{i}^{ob}(0)
		&&\text{if } t = 0
		\\
		&p_{i}^{ob}(t_{l+\frac{1}{2}})\hspace{1.5cm}
		&&\text{if } t\in(t_l,t_{l+1}]\text{ for some }l\in \{0,\dots, N-1\}
	\end{aligned}\right.\hspace{-.21cm}
	\\\label{def ai interpolate}
		a_{i,N,h}(t)
	&\coloneqq \left\{\begin{aligned}
		&a_{i}(0)
		&&\text{if } t = 0
		\\
		&a_{i}(t_{l+\frac{1}{2}})\hspace{1.66cm}
		&&\text{if } t\in(t_l,t_{l+1}]\text{ for some }l\in \{0,\dots, N-1\}
	\end{aligned}\right.\hspace{-.21cm}
\end{align}
for all $i=1,\dots,m$. 
Suppose that $(\{ p_h^l\}_{l=0}^N, \{\bm u_h^{l+\frac{1}{2}}\}_{l=0}^{N-1})$ solves the leapfrog scheme \cref{leapfrog scheme} associated with $\nu_h$. Then, the corresponding interpolations satisfy 
\begin{equation}\label{interpolation scheme state1}
	\left\{\begin{aligned}
		&\int_\Omega(\nu_h\partial_t\Lambda^{p}_{N,h}(t) + \eta\Pi^{p}_{N,h}(t))\phi_h - \Pi^{\bm u}_{N,h}(t)\cdot\nabla\phi_h\d x = \int_\Omega F_{N,h}\phi_h\d x
		\\
		&\begin{aligned}
			&\phantom{\hspace{5cm}}
			&&\text{for all $\phi_h\in\P_{1,D}^h$ and a.e. $t\in I$}
			\\
			&\partial_t\Lambda^{\bm u}_{N,h}(t) + \nabla \Theta^{p}_{N,h}(t) = 0
			&&\text{for a.e. } t\in I
			\\
			&(\Lambda^{p}_{N,h}, \Lambda^{\bm u}_{N,h})(0) = (\Psi_h(p_0), \Phi_h(\nu_h)).
		\end{aligned}
	\end{aligned}\right.
\end{equation}

\begin{lemma}\label{lemma Phi convergence}
	Let \cref{assumption} hold. Furthermore, let ${\{v_h\}_{h> 0} \subset \mathcal V_{ad}}$ converge weakly in $L^2(\Omega)$ towards some $\nu\in \mathcal V_{ad}$ for $h\to 0$. Then, it holds that 
	\begin{equation*}
		\Phi_h(\nu_h) \rightharpoonup \Phi(\nu)\quad\text{weakly in }\bm L^2(\Omega)\text{ as }h\to 0. 
	\end{equation*}
\end{lemma}

\begin{proof}
	For every $h>0$, we set $\Phi_h(\nu_h) = \nabla y_h$ where $y_h\in\P_{1,D}^h$ is according to \cref{definition Phi h} the unique solution to
	\begin{equation}\label{Phi definition new}
		\int_\Omega\nabla y_h\cdot\nabla\phi_h\d x = \int_\Omega(\eta p_0 + \nu_h p_1)\phi_h\d x\quad\forall\phi_h\in\P_{1,D}^h.
	\end{equation}
	Testing \cref{Phi definition new} with $\phi_h = y_h$, along with the Poincaré inequality, we obtain the boundedness of $\{y_h\}_{h> 0}\subset H^1_D(\Omega)$. Therefore, there exists a subsequence (still denoted by the same symbol) and a $y\in H^1_D(\Omega)$ such that 
	\begin{equation}\label{convergence y}
		y_h\rightharpoonup y\quad\text{weakly in }H^1(\Omega)\quad\text{as }h\to 0.
	\end{equation}
	Now, let $\phi\in C^\infty_D(\Omega)$ be arbitrarily fixed. Then, testing \cref{Phi definition new} with $\phi_h = \mathcal I_h\phi\in\P_{1,D}^h$ (see \cref{definition standard interpolation}) and passing to the limit $h\to 0$, we obtain that 
	\begin{align*}
		\int_\Omega\nabla y\cdot\nabla\phi\d x 
		&= \lim_{h\to 0}\int_\Omega\nabla y_h\cdot\nabla\mathcal I_h\phi\d x = \lim_{h\to 0}\int_\Omega(\eta p_0 + \nu_h p_1)\mathcal I_h\phi\d x 
		\\
		&= \int_\Omega(\eta p_0 + \nu p_1)\phi\d x,
	\end{align*}
	where we have used \cref{pih convergence1}, \cref{pih convergence2}, \cref{convergence y}, and the weak convergence $\nu_h\rightharpoonup \nu$ in $L^2(\Omega)$ as $h\to 0$. 
	In conclusion 
	\begin{equation*}
		\int_\Omega\nabla y\cdot\nabla\phi\d x = \int_\Omega(\eta p_0 + \nu p_1)\phi\d x \quad\forall \phi\in C^\infty_D(\Omega).
	\end{equation*}
	Since $H^1_D(\Omega) = \overline{C^\infty_D(\Omega)}^{\|\cdot\|_{H^1(\Omega)}}$, it follows that
	\begin{equation*}
		\int_\Omega\nabla y\cdot\nabla\phi\d x = \int_\Omega(\eta p_0 + \nu p_1)\phi\d x \quad\forall \phi\in H^1_D(\Omega)\quad\underbrace{\Rightarrow}_{\cref{definition Phi}}\quad\nabla y = \Phi(\nu).
	\end{equation*}
\end{proof}

\begin{theorem}\label{theorem interpolations}
	Let \cref{assumption} hold and for every $h>0$, let $\frac{1}{\tau} = N = N(h)\in\mathbb N$ satisfy the {CFL}-condition \cref{cfl}. Furthermore, let ${\{v_h\}_{h> 0} \subset \mathcal V_{ad}}$ converge weakly in $L^2(\Omega)$ towards some $\nu\in \mathcal V_{ad}$ for $h\to 0$. Then, the interpolations of the solution $(\{p_h^l\}_{l=0}^N, \{\bm u_h^{l+\frac{1}{2}}\}_{l=0}^{N-1})\in(\P_{1,D}^h)^{N+1}\times(\DGbm_0^h)^N$ to \cref{leapfrog scheme} satisfy 
	\begin{align*}
		\Lambda^{p}_{N,h}
		&\to p\quad\text{in } C(I, L^2(\Omega))
		&&\text{ as }h\to 0
		\\
		\Theta^{p}_{N,h},\Pi^{p}_{N,h}
		&\to p\quad\text{in } L^\infty(I, L^2(\Omega))
		&&\text{ as }h\to 0,
	\end{align*}
	where $p\in C^1(I, L^2(\Omega))\cap C(I, H^1_D(\Omega))$ is the first component of the unique solution to \cref{system state} associated with $\nu$.
\end{theorem}

\begin{proof}
	By \cref{theorem stability}, along with the definitions of the interpolations (see \cref{def lambda p}-\cref{def pi u}), the families
	\begin{align*}
		\{\Lambda^{p}_{N,h}\}_{h> 0}, \{\Pi^{p}_{N,h}\}_{h> 0}, \{\Theta^{p}_{N,h}\}_{h> 0},\{\partial_t\Lambda^{p}_{N,h}\}_{h> 0}
		&\subset L^\infty(I, L^2(\Omega))
		\\	
		\{\nabla\Theta^{p}_{N,h}\}_{h> 0}, \{\Lambda^{\bm u}_{N,h}\}_{h> 0},\{\Pi^{\bm u}_{N,h}\}_{h> 0},\{\partial_t\Lambda^{\bm u}_{N,h}\}_{h> 0}
		&\subset L^\infty(I, \bm L^2(\Omega))
	\end{align*}
	are bounded. Furthermore, it holds for every $t\in (t_l,t_{l+1}]$ and $l=0,\dots, N-1$ that
	\begin{align*}
		\|\nabla\Lambda^{p}_{N,h}(t)\|_{\bm L^2(\Omega)}
		&\underbrace{=}_{\cref{def lambda p}}\|\nabla(p_h^l + (t - t_l)\delta p_h^{l+\frac{1}{2}})\|_{\bm L^2(\Omega)}
		\\
		&\underbrace{\leq}_{\cref{def deltaphk}} \|\nabla p_h^l\|_{\bm L^2(\Omega)} + \|\nabla(p_h^{l+1} - p_h^l)\|_{\bm L^2(\Omega)},
	\end{align*}
	such that $\{\nabla\Lambda^{p}_{N,h}\}_{h> 0}\subset L^\infty(I, \bm L^2(\Omega))$ is also bounded due to \cref{theorem stability}. On the other hand, by the Aubin-Lions lemma, the following embedding is compact: 
	\begin{equation*}
		\{p\in L^\infty(I, H^1(\Omega)): \partial_t p\in L^\infty(I, L^2(\Omega))\}\overset{c}{\hookrightarrow}C(I, L^2(\Omega)).
	\end{equation*}
	Therefore, we find a subsequence (still denoted by the same symbol) and a $ p\in C(I, L^2(\Omega))$ such that 
	\begin{equation}\label{convergence lambda p}
		\Lambda^{p}_{N,h}\to p\quad\text{in } C(I, L^2(\Omega))\text{ as }h\to 0.
	\end{equation}
	Moreover, it holds that 
	\begin{align*}
		&\|\Lambda^{p}_{N,h} - \Theta^{p}_{N,h}\|_{L^\infty(I,L^2(\Omega))}\!\!\!\underbrace{\leq}_{\cref{def lambda p}, \cref{def theta p}}\!\!\!\max_{l\in\{0,\dots, N-1\}}\tau\|\delta p_h^{l+\frac{1}{2}}\|_{L^2(\Omega)}\!\!\!\underbrace{\leq}_{\text{Thm.}\ref{theorem stability}}\!\!\!\tau C
		\\
		&\hspace{.1cm}\underbrace{\leq}_{\cref{cfl}}c_{cfl}h C\to 0\quad\text{as }h\to 0
	\end{align*}
	and
	\begin{align*}
		&\|\Theta^{p}_{N,h} - \Pi^{p}_{N,h}\|_{L^\infty(I,L^2(\Omega))}\!\!\!\underbrace{\leq}_{\cref{def pi p},\cref{def theta p}}\!\!\!\max_{l\in\{0,\dots, N-1\}}\|p_h^l - p_h^{l+\frac{1}{2}}\|_{L^2(\Omega)}
		\\
		&\underbrace{=}_{\cref{def deltaphk}}\frac{1}{2}\max_{l\in\{0,\dots, N-1\}}\|p_h^{l+1} - p_h^l\|_{L^2(\Omega)}\underbrace{=}_{\cref{def deltaphk}}\max_{l\in\{0,\dots, N-1\}}\frac{\tau}{2}\|\delta p_h^{l+\frac{1}{2}}\|_{L^2(\Omega)}\!\!\underbrace{\leq}_{\text{Thm.}\ref{theorem stability}}\!\!\frac{\tau C}{2}
		\\
		&\hspace{-.04cm}\underbrace{\leq}_{\cref{cfl}}\frac{c_{cfl}h C}{2}\to 0\quad\text{as }h\to 0,
	\end{align*}
	such that, along with \cref{convergence lambda p}, it follows that
	\begin{equation}\label{theta pi p convergence}
		\Theta^{p}_{N,h}, \Pi^{p}_{N,h}\to p\quad\text{in } L^\infty(I, L^2(\Omega))\text{ as }h\to 0.
	\end{equation}
	Due to the Banach-Alaoglu theorem, there exist subsequences (denoted by the same indices), $\tilde p\in L^\infty(I, L^2(\Omega))$, and $\bm v,\bm u,\widehat{\bm u}, \tilde{\bm u}\in L^\infty(I, \bm L^2(\Omega))$ such that 
	\begin{align}
		\label{weak star convergence1}\partial_t\Lambda^{p}_{N,h}
		&\overset{*}{\rightharpoonup} \tilde p
		&&\text{weakly-* in } L^\infty(I, L^2(\Omega))
		\\\notag
		\nabla\Theta^{p}_{N,h}\overset{*}{\rightharpoonup}\bm v,\quad\Lambda^{\bm u}_{N,h}\overset{*}{\rightharpoonup}\bm u,\quad\Pi^{\bm u}_{N,h}\overset{*}{\rightharpoonup}\widehat{\bm u},\quad\partial_t\Lambda^{\bm u}_{N,h}
		&\overset{*}{\rightharpoonup}\tilde{\bm u}
		&&\text{weakly-* in } L^\infty(I, \bm L^2(\Omega)).
	\end{align}
	By standard argumentation, it follows that $\tilde p = \partial_t p$, $\bm v = \nabla p$, and $\tilde{\bm u} = \partial_t\bm u$. Furthermore, as above, we obtain that
	\begin{align}\label{strong convergence u difference}
		&\|\Lambda^{\bm u}_{N,h} - \Pi^{\bm u}_{N,h}\|_{L^\infty(I,L^2(\Omega))}\!\!\!\underbrace{\leq}_{\cref{def lambda u},\cref{def pi u}}\!\!\max_{l\in\{0,\dots, N-1\}}\!\left(\|\bm u_h^{l-\frac{1}{2}} - \bm u_h^{l+\frac{1}{2}}\|_{L^2(\Omega)} + \tau\|\delta \bm u_h^l\|_{L^2(\Omega)}\right)
		\\\notag
		&\hspace{.14cm}\underbrace{=}_{\cref{def deltauhk}}\max_{l\in\{0,\dots, N-1\}}2\tau\|\delta\bm u_h^l\|_{L^2(\Omega)}\!\!\underbrace{\leq}_{\text{Thm.}\ref{theorem stability}}\!\!2\tau C\underbrace{\leq}_{\cref{cfl}}\! 2c_{cfl}h C\to 0,\,\,\text{as }h\to 0,
	\end{align}
	such that, along with \cref{weak star convergence1}, it follows $\bm u = \widehat{\bm u}$. Next, let us prove that $( p, \bm u)$ is the unique solution to \cref{system state} associated with $\nu$. Let $t\in I\setminus\{0\}$ be arbitrarily given. Then, for every $h>0$, there exists some $l(h)\in\{0,\dots,N(h)-1\}$ such that $t\in(t_{l(h)},t_{l(h)+1}]$. In particular, we have that $t_{l(h)+\frac{1}{2}}\to t$ as $h\to 0$, and since $F\in W^{1,1}(I, L^2(\Omega))$, this implies that
	\begin{equation*}
		\|F_{N,h}(t)- F(t)\|_{L^2(\Omega)}\underbrace{=}_{\cref{def F interpolate}} \|F(t_{l(h)+\frac{1}{2}}) - F(t)\|_{L^2(\Omega)}\leq\int_{t}^{t_{l(h)+\frac{1}{2}}}\!\|\partial_tF(s)\|_{L^2(\Omega)}\d s\to 0
	\end{equation*}
	as $h\to 0$. Along with $\|F_{N,h}(t)\|_{L^2(\Omega)}\leq \|F\|_{C(I, L^2(\Omega))}$ for every $t\in I$ (see \cref{def F interpolate}), Lebesgue's dominated convergence theorem implies that
	\begin{equation}\label{convergence F}
		F_{N,h}\to F\text{ in }L^2(I, L^2(\Omega))\quad\text{as }h\to 0.
	\end{equation}
	Since $\{\nu_h\}_{h> 0}$ converges weakly in $L^2(\Omega)$ towards $\nu$, along with \cref{convergence lambda p}, it holds that for every $\phi\in L^2(I, C^\infty(\Omega))$ that
	\begin{align}\label{nuh pnh convergence1}
		&|(\nu_h\Lambda^{p}_{N,h} - \nu p, \phi)_{L^2(I, L^2(\Omega))}|
		\\\notag
		&\leq |(\nu_h(\Lambda^{p}_{N,h} - p), \phi)_{L^2(I, L^2(\Omega))}| + |((\nu_h - \nu)\phi, p)_{L^2(I, L^2(\Omega))}|
		\\\notag
		&\leq|((\nu_h - \nu)\phi,(\Lambda^{p}_{N,h} - p))_{L^2(I, L^2(\Omega))}| + |((\Lambda^{p}_{N,h} - p),\nu\phi)_{L^2(I, L^2(\Omega))}| \\\notag
 		&\quad+ |((\nu_h - \nu)\phi, p)_{L^2(I, L^2(\Omega))}|\to 0\quad\text{as }h\to 0.
	\end{align}
	Due to \cref{convergence lambda p} and $\{\nu_h\}_{h>0}\subset\mathcal V_{ad}$, the sequence $\{\nu_h\Lambda^{p}_{N,h}\}_{h> 0}\subset L^2(I, L^2(\Omega))$ is bounded. Thus, along with the density $C^\infty(I, L^2(\Omega))\subset L^2(I, L^2(\Omega))$, \cref{nuh pnh convergence1} implies that
	\begin{equation}\label{nuh lambda p convergence}
		\nu_h\Lambda^{p}_{N,h}\rightharpoonup \nu p\quad\text{weakly in }L^2(I, L^2(\Omega))\quad\text{as }h\to 0.
	\end{equation}
	Moreover, for every $\phi\in C^\infty_0(I, L^2(\Omega))$, it holds for $h\to 0$ that 
	\begin{align}\label{nuh pnh convergence2}
		&(\nu_h\partial_t\Lambda^{p}_{N,h}, \phi)_{L^2(I, L^2(\Omega))} = -(\nu_h\Lambda^{p}_{N,h}, \partial_t\phi)_{L^2(I, L^2(\Omega))}
		\\\notag
		&\underbrace{\to}_{\cref{nuh lambda p convergence}} -(\nu p, \partial_t\phi)_{L^2(I, L^2(\Omega))}= (\nu\partial_t p, \phi)_{L^2(I, L^2(\Omega))}.
	\end{align}
	Due to \cref{weak star convergence1} and $\{\nu_h\}_{h>0}\subset\mathcal V_{ad}$, the sequence $\{\nu_h\partial_t\Lambda^{p}_{N,h}\}_{h> 0}\subset L^2(I, L^2(\Omega))$ is also bounded. Therefore, along with the density $C^\infty_0(I, L^2(\Omega)\subset L^2(I, L^2(\Omega))$, \cref{nuh pnh convergence2} implies that
	\begin{equation}\label{nuh lambda p weak convergence}
		\nu_h\partial_t\Lambda^{p}_{N,h}\rightharpoonup \nu\partial_t p\quad\text{weakly in }L^2(I, L^2(\Omega))\quad\text{as }h\to 0.
	\end{equation}
	Using the $\P^h_1$-interpolation operator $\mathcal I_h$ (see \cref{definition standard interpolation}), we deduce for every $\phi\in C^\infty_D(\Omega)$ and $\psi\in C^\infty_0(I)$ that 
	\begin{align*}
		&\int_I\int_\Omega(\nu\partial_t p(t) + \eta p(t))\phi - \bm u(t)\cdot\nabla\phi\d x\psi(t)\d t
		\\
		&=\lim_{h\to 0}\int_I\int_\Omega(\nu_h\partial_t \Lambda^{p}_{N,h}(t) + \eta \Pi^{p}_{N,h}(t))\mathcal I_h\phi - \Pi^{\bm u}_{N,h}(t)\cdot\nabla(\mathcal I_h\phi)\d x \psi(t)\d t
		\\
		&=\lim_{h\to 0}\int_I\int_\Omega F_{N,h}(t)\mathcal I_h\phi\d x\psi(t)\d t=\int_I\int_\Omega F(t)\phi\d x\psi(t)\d t,
	\end{align*}
 where we have used \cref{pih convergence2},\cref{interpolation scheme state1},\cref{theta pi p convergence},\cref{weak star convergence1},\cref{convergence F}, and \cref{nuh lambda p weak convergence}. Since $H^1_D(\Omega) = \overline{C^\infty_D(\Omega)}^{\|\cdot\|_{H^1(\Omega)}}$, it follows for all $\phi\in H^1_D(\Omega)$ and almost every $t\in I$ that 
	\begin{align}\label{conv eps h}
		&\int_\Omega(\nu\partial_t p(t) + \eta p(t))\phi - \bm u(t)\cdot\nabla\phi\d x
		=\int_\Omega F(t)\phi\d x.
	\end{align}
	Furthermore, by \cref{weak star convergence1} and \cref{interpolation scheme state1}, we obtain that
	\begin{equation}\label{state equation in proof}
		\partial_t\bm u = -\nabla p.	
	\end{equation}
	Regarding the initial conditions, it holds that
	\begin{equation}\label{state equation in proof2}
		p(0)\underbrace{=}_{\cref{convergence lambda p}}\lim_{h\to 0}\Lambda^{p}_{N,h}(0) \underbrace{=}_{\cref{interpolation scheme state1}} \lim_{h\to 0}\Psi_h(p_0)\,\underbrace{=}_{\cref{convergence psi}}\, p_0\quad\text{in }L^2(\Omega).
	\end{equation}
	Let $\bm\phi\in \bm L^2(\Omega)$ be arbitrarily fixed and $\xi\in C^\infty(I)$ such that $\xi(T) = 0$ and $\xi(0) = 1$. We define $[\bm\phi\xi]\in C^\infty(I, \bm L^2(\Omega))$ by $[\bm\phi\xi](t)(x) = \bm\phi(x)\xi(t)$ for all $t\in I$ and a.e. $x\in\Omega$. Then, along with the integration by parts formula, it follows that
	\begin{align*}
		&(\bm\phi,\Phi(\nu))_{\bm L^2(\Omega)}\underbrace{=}_{\text{Lem.}\ref{lemma Phi convergence}}\lim_{h\to 0}(\bm\phi,\Phi_h(\nu_h))_{\bm L^2(\Omega)}\underbrace{=}_{\cref{interpolation scheme state1}}\lim_{h\to 0}(\bm\phi,\Lambda_{N,h}^{\bm u}(0))_{\bm L^2(\Omega)} 
		\\
		&=\lim_{h\to 0}\left(-\int_I(\bm\phi,\partial_t\Lambda_{N,h}^{\bm u}(t))_{\bm L^2(\Omega)}\xi(t)\d t - \int_I(\bm\phi,\Lambda_{N,h}^{\bm u}(t))_{\bm L^2(\Omega)}\partial_t\xi(t)\d t \right)
		\\
		&\hspace{-.18cm}\underbrace{=}_{\cref{weak star convergence1}}
		-\int_I(\bm\phi,\partial_t\bm u(t))_{\bm L^2(\Omega)}\xi(t)\d t - \int_I(\bm\phi,\bm u(t))_{L^2(\Omega)}\partial_t\xi(t)\d t 
		\\
		&= (\bm\phi,\bm u(0))_{\bm L^2(\Omega)}.
	\end{align*}
	Since $\bm\phi\in\bm L^2(\Omega)$ was chosen arbitrarily, this implies $\bm u(0) = \Phi(\nu)$. Finally, from \cref{conv eps h}-\cref{state equation in proof2}, we conclude that 
 \begin{equation*}
 (p,\bm u)\in W^{1,\infty}(I, L^2(\Omega))\cap L^2(I, H^1_D(\Omega))\times W^{1,\infty}(I, \bm L^2(\Omega))\cap L^2(I, \bm H_N(\Div,\Omega)) 
 \end{equation*}
 satisfies
	\begin{equation}\label{system weak}
		\left\{\begin{aligned}
			&\nu\partial_t p + \Div\bm u + \eta p = F
			&&\text{in }I\times\Omega
			\\
			&\partial_t\bm u + \nabla p = 0
			&&\text{in }I\times\Omega
			\\
			&p = 0 
			&& \text{on }I\times \Gamma_D
			\\
			&\bm{u} \cdot\bm{n} = 0 
			&&\text{on }I\times \Gamma_N
			\\
			&(p, \bm u)(0) = (p_0,\Phi(\nu))
			&&\text{in }\Omega.
		\end{aligned}\right.
	\end{equation}
	Thus, $(p,\bm u)$ is the strong solution to \cref{system state} associated with $\nu$. Since the unique classical and strong solutions to \cref{system state} coincide, the claim follows.
\end{proof}

To prove our final result, that is \cref{theorem convergence}, we need the following lemma:

\begin{lemma}\label{lemma convergence objective}
	Let \cref{assumption} hold and for every $h>0$, let $\frac{1}{\tau} = N = N(h)\in\mathbb N$ satisfy the {CFL}-condition \cref{cfl}. Furthermore, let $\{\nu_h\}_{h>0}$ with $\nu_h\in\mathcal V_{ad}^h$ for all $h>0$ be given. Then, it holds that 
	\begin{equation}\label{first claim}
		|J(\nu_h) - J_h(\nu_h)|\to 0\quad\text{as }h\to 0.
	\end{equation}
	If additionally $\{\nu_h\}_{h>0}$ converges (strongly) in $L^2(\Omega)$ towards $\nu\in\mathcal V_{ad}$ as $h\to 0$, then it holds that 
	\begin{equation}\label{second claim}
		|J(\nu) - J_h(\nu_h)|\to 0\quad\text{as }h\to 0.
	\end{equation}
\end{lemma}

\begin{proof}
	Let $p_h\in C^1(I, L^2(\Omega))\cap C(I, H^1_D(\Omega))$ denote the first component of the unique solution to \cref{system state} associated with $\nu_h$. Furthermore, for every $h>0$, let $(\{p_h^l\}_{l=0}^N,\{\bm u_h^{l+\frac{1}{2}}\}_{l=0}^{N-1})\in(\P_{1,D}^h)^{N+1}\times(\DGbm_0^h)^N$ denote the unique solution to \cref{leapfrog scheme} and let $\Pi^{p}_{N,h}$ denote the corresponding interpolate as in \cref{def pi p}. Note that $\mathcal V_{ad}^h\subset\mathcal V_{ad}$. Moreover, $\mathcal V_{ad}$ is bounded, closed, and convex in $L^2(\Omega)$, and thus, it is weakly sequentially compact in $L^2(\Omega)$. Consequently, there exists a subsequence of $\{\nu_h\}_{h>0}$ (still denoted by the same symbol) and a $\tilde\nu\in\mathcal V_{ad}$ such that $\{\nu_h\}_{h>0}$ convergences weakly in $L^2(\Omega)$ towards $\tilde\nu$ as $h\to 0$. Thus, \cref{theorem interpolations} implies that 
	\begin{equation}\label{convergence widehatp}
		\Pi^{p}_{N,h}\to\tilde p \quad\text{in }L^\infty(I, L^2(\Omega))\text{ as } h\to 0,
	\end{equation}
	where $\tilde p\in C^1(I, L^2(\Omega))\cap C(I, H^1_D(\Omega))$ denotes first component of the unique solution to \cref{system state} associated with $\tilde\nu$. Furthermore, the solution operator $S_p\colon L^2(\Omega)\to L^2(I, L^2(\Omega)), \nu\mapsto p$ to \cref{system state} is weak-strong continuous (see the proof of \cite[Theorem 2.4]{ammann23}) such that 
	\begin{equation}\label{convergence ph}
		p_h\to\tilde p \quad\text{in }L^2(I, L^2(\Omega))\quad\text{ as } h\to 0.
	\end{equation}
	Since $p_i^{ob}\in W^{1,1}(I, L^2(\Omega))$, as in \cref{convergence F}, it follows that 
	\begin{equation}\label{convergence piob}
		 p_{i,N,h}^{ob}\to p^{ob}_i\quad\text{in }L^2(I, L^2(\Omega))\quad\text{as }h\to 0\quad\forall i\in\{1,\dots,m\}.
	\end{equation}
	Now, it holds that
	\begin{align}\label{J - Jh}
		&|J(\nu_h) - J_h(\nu_h)|
		\\\notag
		&\underbrace{=}_{\cref{P},\cref{P discretized reduced}} \Bigg|\frac{1}{2} \sum_{i=1}^m \int_I\int_{\Omega}a_i(t)(p_h(t)-p^{ob}_i(t))^2\d x\d t 
		\\\notag
		&\hspace{1.3cm}- \frac{\tau}{2} \sum_{i=1}^m\sum_{l=0}^{N-1}\int_{\Omega}a_i(t_{l+\frac{1}{2}})(p_h^{l+\frac{1}{2}}-p^{ob}_{i}(t_{l+\frac{1}{2}}))^2\d x\Bigg|
		\\\notag
		&\underbrace{=}_{\cref{def pi p},\cref{def ai interpolate},\cref{def pob interpolate}}\Bigg|\frac{1}{2} \sum_{i=1}^m \int_I\int_{\Omega}a_i(t)(p_h(t)-p^{ob}_i(t))^2\d x\d t 
		\\\notag
		&\hspace{2.2cm}- \frac{1}{2} \sum_{i=1}^m \int_I\int_{\Omega}a_{i,N,h}(t)(\Pi^{ p}_{N,h}(t)- p_{i,N,h}^{ob}(t))^2\d x\d t \Bigg|
		\\\notag
		&\leq\frac{1}{2}\sum_{i=1}^m\|a_i\|_{C(I,L^\infty(\Omega))}\int_I\int_{\Omega}|(p_h(t)-p^{ob}_i(t))^2 - (\Pi^{ p}_{N,h}(t) - p_{i,N,h}^{ob}(t))^2|\d x\d t
		\\\notag
		&\quad+\frac{1}{2}\sum_{i=1}^m\|a_i - a_{i,N,h}\|_{L^1(I, L^\infty(\Omega))}\|\Pi_{N,h}^p - p_{i,N,h}^{ob}\|_{L^\infty(I, L^2(\Omega))}^2
		\\\notag
		&\leq\frac{1}{2}\sum_{i=1}^m\|a_i\|_{C(I,L^\infty(\Omega))}\int_I\int_{\Omega}|(p_h(t)-p^{ob}_i(t) - \Pi^{ p}_{N,h}(t) + p_{i,N,h}^{ob}(t))\cdot
		\\\notag
		&\hspace{4.6cm}(p_h(t) - p^{ob}_i(t) + \Pi^{p}_{N,h}(t) - p_{i,N,h}^{ob}(t))|\d x\d t
		\\\notag
		&\quad+\sum_{i=1}^m\|a_i - a_{i,N,h}\|_{L^1(I, L^\infty(\Omega))}(\|\Pi_{N,h}^p\|_{L^\infty(I, L^2(\Omega))}^2 + \|p_{i,N,h}^{ob}\|_{L^\infty(I, L^2(\Omega))}^2)
		\\\notag
		&\leq\frac{1}{2}\sum_{i=1}^m\Big(\|a_i\|_{C(I,L^\infty(\Omega))}(\|p_h - \Pi^{ p}_{N,h}\|_{L^2(I,L^2(\Omega))} + \|p^{ob}_i - p_{i,N,h}^{ob}\|_{L^2(I,L^2(\Omega))})\cdot
		\\\notag
		&\hspace{1.6cm}(\|p_h\|_{L^2(I,L^2(\Omega))} + \|p^{ob}_i\|_{L^2(I,L^2(\Omega))} + \|\Pi^{p}_{N,h}\|_{L^2(I,L^2(\Omega))} 
		\\\notag
		&\hspace{1.7cm}+ \| p_{i,N,h}^{ob}\|_{L^2(I,L^2(\Omega))})\Big)
		\\\notag
		&\quad+\sum_{i=1}^m\|a_i - a_{i,N,h}\|_{L^1(I, L^\infty(\Omega))}(\|\Pi_{N,h}^p\|_{L^\infty(I, L^2(\Omega))}^2 + \|p_{i,N,h}^{ob}\|_{L^\infty(I, L^2(\Omega))}^2).
	\end{align}
	Since $a_i\in C^1(I, L^\infty(\Omega))$, as in \cref{convergence F}, it follows that 
	\begin{equation}\label{ai convergence}
		a_{i,N,h}\to a_i\quad\text{in }L^1(I, L^\infty(\Omega))	\quad\text{as }h\to 0\quad\forall i\in \{1,\dots,m\}.
	\end{equation} 
	Moreover, since $\{\nu_h\}_{h>0}\subset L^\infty(\Omega)$ is bounded, $\{p_h\}_{h>0}\subset L^2(I, L^2(\Omega))$ is bounded due to \cite[Lemma 2.2]{ammann23}. Furthermore, $\{\Pi^{p}_{N,h}\}_{h>0}\subset L^\infty(I, L^2(\Omega))$ is bounded due to \cref{convergence widehatp} and $\{ p_{i,N,h}^{ob}\}_{h>0}\subset L^\infty(I, L^2(\Omega))$ is bounded due to \cref{def pob interpolate}. Consequently \cref{convergence widehatp}-\cref{ai convergence} imply \cref{first claim}. Now, let additionally $\{\nu_h\}_{h>0}$ converge strongly in $L^2(\Omega)$ towards $\nu\in\mathcal V_{ad}$. Then, due to \cref{theorem interpolations}, we have that 
	\begin{equation}\label{pi p convergence}
		\Pi^{p}_{N,h}\to p\quad\text{in } L^\infty(I, L^2(\Omega))\text{ as }h\to 0,
	\end{equation}
	where $p\in C^1(I, L^2(\Omega))\cap C(I, H^1_D(\Omega))$ is the first component of the unique solution to \cref{system state} associated with $\nu$. Analogously to \cref{J - Jh}, we deduce that 
	\begin{align*}
		&|J(\nu) - J_h(\nu_h)|
		\\
		&\underbrace{=}_{\cref{P},\cref{P discretized reduced}} \Bigg|\frac{1}{2} \sum_{i=1}^m \int_I\int_{\Omega}a_i(t)(p(t)-p^{ob}_i(t))^2\d x\d t + \frac{\lambda}{2}\|\nu\|_{L^2(\Omega)}^2 
		\\
		&\hspace{1.3cm}- \frac{\tau}{2} \sum_{i=1}^m\sum_{l=0}^{N-1}\int_{\Omega}a_i(t_{l+\frac{1}{2}})(p_h^{l+\frac{1}{2}}-p^{ob}_{i}(t_{l+\frac{1}{2}}))^2\d x- \frac{\lambda}{2}\|\nu_h\|_{L^2(\Omega)}^2\Bigg|
		\\
		&\hspace{.35cm}\leq\frac{1}{2}\sum_{i=1}^m\Big(\|a_i\|_{C(I, L^\infty(\Omega))}(\|p - \Pi^{p}_{N,h}\|_{L^2(I,L^2(\Omega))} + \|p^{ob}_i - p_{i,N,h}^{ob}\|_{L^2(I,L^2(\Omega))})\cdot
		\\
		&\hspace{2cm}(\|p\|_{L^2(I,L^2(\Omega))} + \|p^{ob}_i\|_{L^2(I,L^2(\Omega))} + \|\Pi^{p}_{N,h}\|_{L^2(I,L^2(\Omega))} 
		\\
		&\hspace{2.1cm}+ \| p_{i,N,h}^{ob}\|_{L^2(I,L^2(\Omega))})\Big)
		\\
		&\hspace{.79cm}+ \sum_{i=1}^m\|a_i - a_{i,N,h}\|_{L^1(I, L^\infty(\Omega))}(\|\Pi_{N,h}^p\|_{L^2(I, L^2(\Omega))}^2 + \|p_{i,N,h}^{ob}\|_{L^2(I, L^2(\Omega))}^2) 
		\\\notag
		&\hspace{.79cm}+ \frac{\lambda}{2} |\|\nu\|_{L^2(\Omega)}^2 - \|\nu_h\|_{L^2(\Omega)}^2|,
	\end{align*}
	where the right-hand side converges to $0$, due to \cref{convergence piob},\cref{ai convergence}, and \cref{pi p convergence}.
\end{proof}

\begin{theorem}\label{theorem convergence}
	Let \cref{assumption} hold and for every $h>0$, let $\frac{1}{\tau} = N = N(h)\in\mathbb N$ satisfy the {CFL}-condition \cref{cfl}. Furthermore, let $\overline\nu\in\mathcal V_{ad}$ be a locally optimal solution to \cref{P} such that the following quadratic growth condition holds:
	\begin{equation}\label{growth3}
		\exists\, \sigma, \delta >0\,:\quad J(\nu)\geq J(\overline\nu) + \delta\|\nu-\overline\nu\|_{L^2(\Omega)}^2\quad\forall \nu\in\mathcal V_{ad}\quad\text{with}\quad\|\nu - \overline\nu\|_{L^2(\Omega)}^2\leq\sigma.
	\end{equation}
	Then, there exists a sequence $\{\overline\nu_h\}_{h>0}$ with $\nu_h\in\mathcal V_{ad}^h$ such that 
	\begin{equation*}
		\overline\nu_h\to\overline\nu\quad\text{in }L^2(\Omega)\text{ as }h\to 0,
	\end{equation*}
	and $\overline \nu_h \in\mathcal V_{ad}^h$ is a locally optimal solution to \cref{P discretized reduced} for all sufficiently small $h>0$. 
\end{theorem}

\begin{remark}
	Note that the quadratic growth condition \cref{growth3} is reasonable since it can be obtained by assuming suitable regularity and compatibility conditions and a sufficient second-order optimality condition (see \cite[Theorem 4.6]{ammann23}). 
\end{remark}

\begin{proof}
	Inspired by the idea by Casas and Tr\"oltzsch \cite{casastroeltzsch02}, for $\mathcal V_{ad}^{h,\sigma}\coloneqq\{\nu_h\in\mathcal V_{ad}^h\,|\,\|\nu_h - \overline\nu\|_{L^2(\Omega)}\leq\sigma\}$, we consider the minimization problem 
	\begin{equation}\label{P discretized reduced sigma}\tag{$\P_h^\sigma$}
		\min_{\nu_h\in\mathcal V_{ad}^{h,\sigma}}J_h(\nu_h),
	\end{equation}
	where the discrete reduced cost functional $J_h$ is given as in \cref{P discretized reduced}. With $Q_h\colon L^2(\Omega)\to\DG_0^h$ we denote the standard $L^2(\Omega)$-orthogonal projection operator onto $\DG_0^h$ that is given by 
	\begin{equation}
		(Q_h v)(x)=\sum_{T\in\mathcal T_h}\chi_T(x)\frac{1}{|T|}\int_T v(y)\d y\quad\forall v\in L^2(\Omega) \quad\forall x\in \overline\Omega,
	\end{equation}
	where $\chi_T$ denotes the characteristic function of $T$. Then, there exists a constant $c>0$, independent of $h$, such that 
	\begin{equation*}
		\|v - Q_h v\|_{L^2(\Omega)}\leq c h\|v\|_{H^1(\Omega)}\quad\forall v\in H^1(\Omega)
	\end{equation*}
	(see \cite[Prop. 1.135]{ern04}). Now, let $v\in L^2(\Omega)$ and $\epsilon>0$ be arbitrarily given. By the density $H^1(\Omega)\subset L^2(\Omega)$, there exists $v_\epsilon\in H^1(\Omega)$ with $\|v_\epsilon -v\|_{L^2(\Omega)}\leq\frac{\epsilon}{3}$. Then, it holds that 
	\begin{align*}
		\|v - Q_hv\|_{L^2(\Omega)}
		&\leq \|v - v_\epsilon\|_{L^2(\Omega)} + \|v_\epsilon - Q_hv_\epsilon\|_{L^2(\Omega)} + \|Q_hv_\epsilon - Q_hv\|_{L^2(\Omega)}
		\\
		&\leq 2\|v - v_\epsilon\|_{L^2(\Omega)} + ch\|v_\epsilon\|_{H^1(\Omega)}\leq \epsilon\quad\forall h\in \left(0,\frac{\epsilon}{3c\|v_\epsilon\|_{H^1(\Omega)}}\right),
	\end{align*}
	which implies that $Q_hv\to v$ in $L^2(\Omega)$ as $h\to 0$ for every $v\in L^2(\Omega)$. Furthermore, $Q_h$ maps $\mathcal V_{ad}$ into $\mathcal V_{ad}^h$ since for every $\nu\in\mathcal V_{ad}$ and every $x\in\overline\Omega$ it holds that
	\begin{equation*}
		\nu_- = \sum_{T\in\mathcal T_h}\chi_T(x)\frac{1}{|T|}\int_T v_-\d y \leq (Q_hv)(x) \leq \sum_{T\in\mathcal T_h}\chi_T(x)\frac{1}{|T|}\int_T v_+\d y = \nu_+.
	\end{equation*}
	Therefore, the sequence $\{\tilde\nu_h\}_{h>0}\coloneqq \{Q_h\overline\nu\}_{h>0}$ satisfies 
	\begin{equation}\label{nuh to nu}
		\tilde\nu_h\in\mathcal V_{ad}^h\quad\forall h>0\quad\text{and}\quad\tilde\nu_h\to \overline\nu\quad\text{in }L^2(\Omega)\quad\text{as }h\to 0. 
	\end{equation} 
	Consequently, there exists $\overline h>0$ such that $\tilde\nu_h\in\mathcal V_{ad}^{h,\sigma}$ for all $h\in(0,\overline h]$. In particular, $\mathcal V_{ad}^{h,\sigma}$ is non-empty for all $h\in(0,\overline h]$. Since $\mathcal V_{ad}^{h,\sigma}\subset\DG_0^h$ is compact and $J_h\colon L^2(\Omega)\supset\mathcal V_{ad}^{h,\sigma}\to\mathbb R$ is continuous (see the proof of \cref{theorem existence}), for every $h\in(0,\overline h]$, the auxiliary problem \cref{P discretized reduced sigma} admits at least one minimizer $\overline\nu_h\in\mathcal V_{ad}^{h,\sigma}$ due to the Weierstrass theorem. Then, by \cref{growth3}, we obtain for all $h\in(0,\overline h]$ that 
	\begin{equation}\label{ieq strong convergence}
		\delta\|\overline\nu_h-\overline\nu\|_{L^2(\Omega)}^2\leq J(\overline\nu_h) - J(\overline\nu) = [J(\overline\nu_h) - J_h(\overline\nu_h)] + [J_h(\overline\nu_h) - J_h(\tilde\nu_h)]+ [J_h(\tilde\nu_h)- J(\overline\nu)].
	\end{equation}
	The first and the third summand at the right-hand side of \cref{ieq strong convergence} converge to zero due to \cref{lemma convergence objective} as $h\to 0$. The second summand in \cref{ieq strong convergence} is non-positive since $\overline\nu_h\in\mathcal V_{ad}^{h,\sigma}$ is a minimizer of \cref{P discretized reduced sigma} and $\tilde\nu_h\in\mathcal V_{ad}^{h,\sigma}$ for all $h\in(0,\overline h]$. Therefore, \cref{ieq strong convergence} implies that 
	\begin{equation}\label{strong convergence}
		\overline\nu_h\to\overline\nu\quad\text{strongly in }L^2(\Omega)\text{ as }h\to 0.
	\end{equation}
	It remains to show that $\overline\nu_h\in\mathcal V_{ad}^h$ is a local minimizer to \cref{P discretized reduced}. For this purpose, let $\nu_h\in\mathcal V_{ad}^h$ with $\|\nu_h - \overline\nu_h\|_{L^2(\Omega)}\leq\frac{\sigma}{2}$ be arbitrarily given. According to \cref{strong convergence}, there exists $\widehat h>0$ such that $\|\overline\nu_h - \overline\nu\|_{L^2(\Omega)}\leq\frac{\sigma}{2}$ for all $h\in(0,\min\{\overline h,\widehat h\}]$. Therefore, it follows that
	\begin{equation*}
		\|\nu_h - \overline\nu\|_{L^2(\Omega)}\leq \|\nu_h - \overline\nu_h\|_{L^2(\Omega)} + \|\overline\nu_h - \overline\nu\|_{L^2(\Omega)}\leq\frac{\sigma}{2} + \frac{\sigma}{2} = \sigma\quad\forall h\in (0,\min\{\overline h,\widehat h\}].
	\end{equation*}
	In other words, $\nu_h\in\mathcal V_{ad}^{h,\sigma}$ for all $h\in (0,\min\{\overline h,\widehat h\}]$. Since $\overline\nu_h$ is a minimizer to \cref{P discretized reduced sigma}, it follows that $J_h(\overline\nu_h)\leq J_h(\nu_h)$ for all $\nu_h\in\mathcal V_{ad}^h$ with $\|\nu_h - \overline\nu_h\|_{L^2(\Omega)}\leq\frac{\sigma}{2}$. In other words, $\overline\nu_h$ is a local minimizer to \cref{P discretized reduced} for all $h\in (0,\min\{\overline h,\widehat h\}]$. 
\end{proof}

\section{Numerical test}
	
  In this section, we present a short numerical test to illustrate the performance of the proposed finite element discretization. We set $T= 2$ and $\Omega := (0,2)\times(0,1)$ with the Neumann boundary part $\Gamma_N=[0,2]\times\{1\}$ and the Dirichlet part $\Gamma_D= \partial\Omega\setminus\Gamma_N$. The damping term $\eta$, the source signal $f$, and the receivers $a_i$ for $i=1,\dots,30$ are exaclty specified as in our previous work \cite[section 5]{ammann23}. We consider a desired paramter $\nu_d\colon\overline\Omega\to\mathbb R$ given by (see \cref{fig2:a})
  \begin{equation*}
	\nu_d(x)\coloneqq \left\{\begin{aligned}
		&1.2
		&&\text{if } x\in[58/64,70/64]\times[39/64,51/64]
		\\
		&1.4
		&&\text{if } x\in[26/64,38/64]\times[26/64,38/64]
		\\
		&1.6
		&&\text{if } x\in[90/64,102/64]\times[20/64,32/64]
		\\
		&1
		&&\text{else.}
	\end{aligned}\right.
  \end{equation*}
  The observation wave information is specified by the solution to the forward problem \cref{system state} for $\nu = \nu_d$ under the deterministic noise model
  \begin{equation*}
	p^{ob}\coloneqq S_p(\nu_d) + \mu
  \end{equation*}
  with the noise level is $l=2\%$ (cf. \cite[section 5]{ammann23}). Furthermore, we choose the regularization paramter $\lambda\coloneqq 0.001$ and the lower and upper bounds $\nu_{\min}=1$ and $\nu_{\max} =1.6$. For our computational implementation, we utilize the programming language \texttt{Python} (version 3.6.9) and the package \texttt{DOLFIN} from the open-source computing platform \texttt{FEniCS} (cf. \cite{logg10,logg12}). We set the discretization paremeter $h = 1/64$. Since our numerical tests show the necessity to satisfy the CFL-condition \cref{cfl} with $c_{cfl}=1/6$, we set $N= 384$. To solve the discrete optimal control problem, we choose an SQP method (see \cite{ammann24}) with the initial value $\nu_0\equiv 1$. Using the proposed finite element discretization, the SQP method reasonably reconstructs the parameter $\nu_d$ after 16 iterations (see \cref{fig2:b}).

  \begin{figure}[htbp]
	\centering
	\captionsetup{justification=centering}
	\subfloat[Reference $\nu_d$.]{
		\label{fig2:a}\includegraphics[trim = 12.9cm 11.8cm 12.9cm 11.8cm, clip,width=0.45\textwidth]{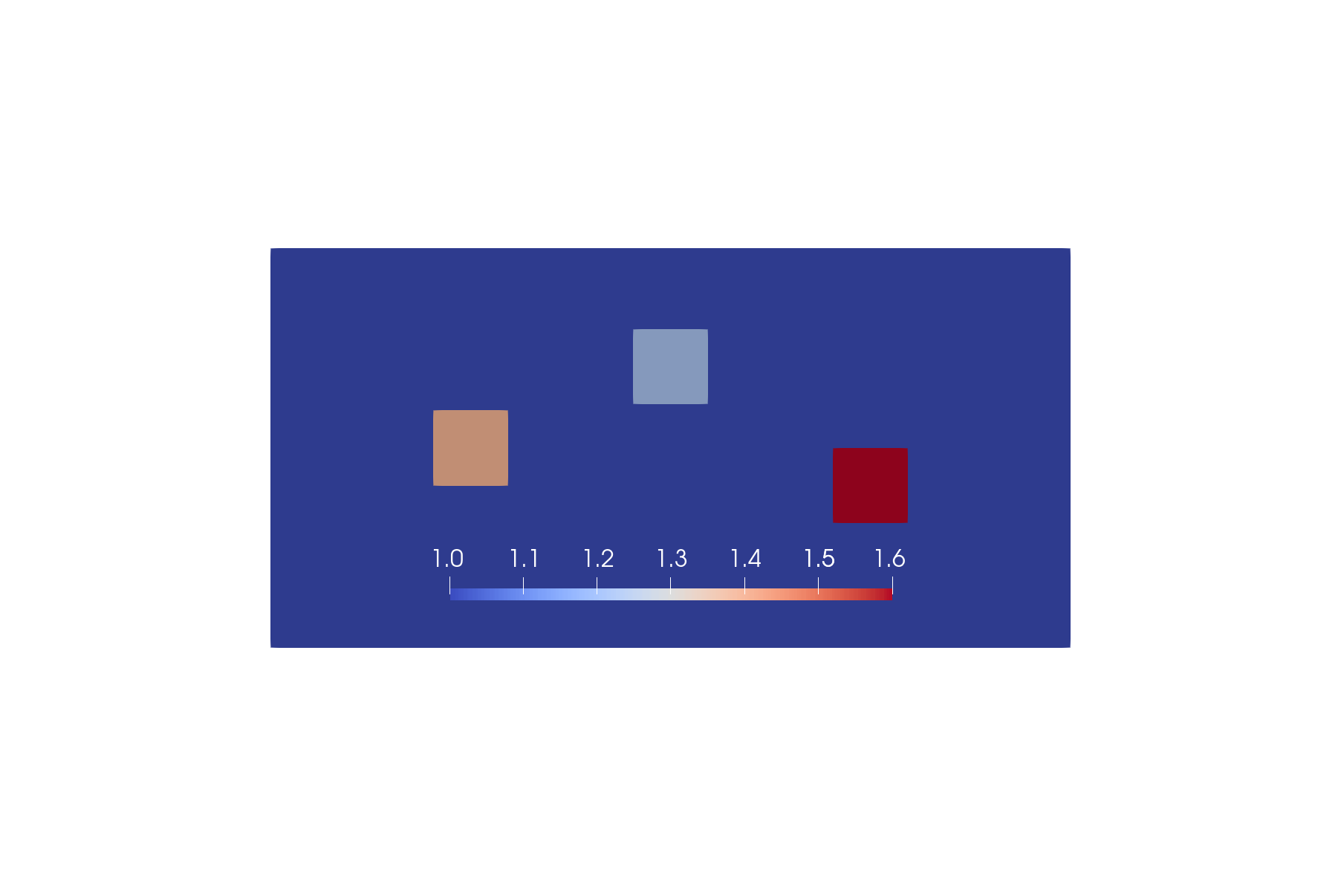}
	}
	\subfloat[Computed solution.]{
		\label{fig2:b}\includegraphics[trim = 12.9cm 11.8cm 12.9cm 11.8cm, clip,width=0.45\textwidth]{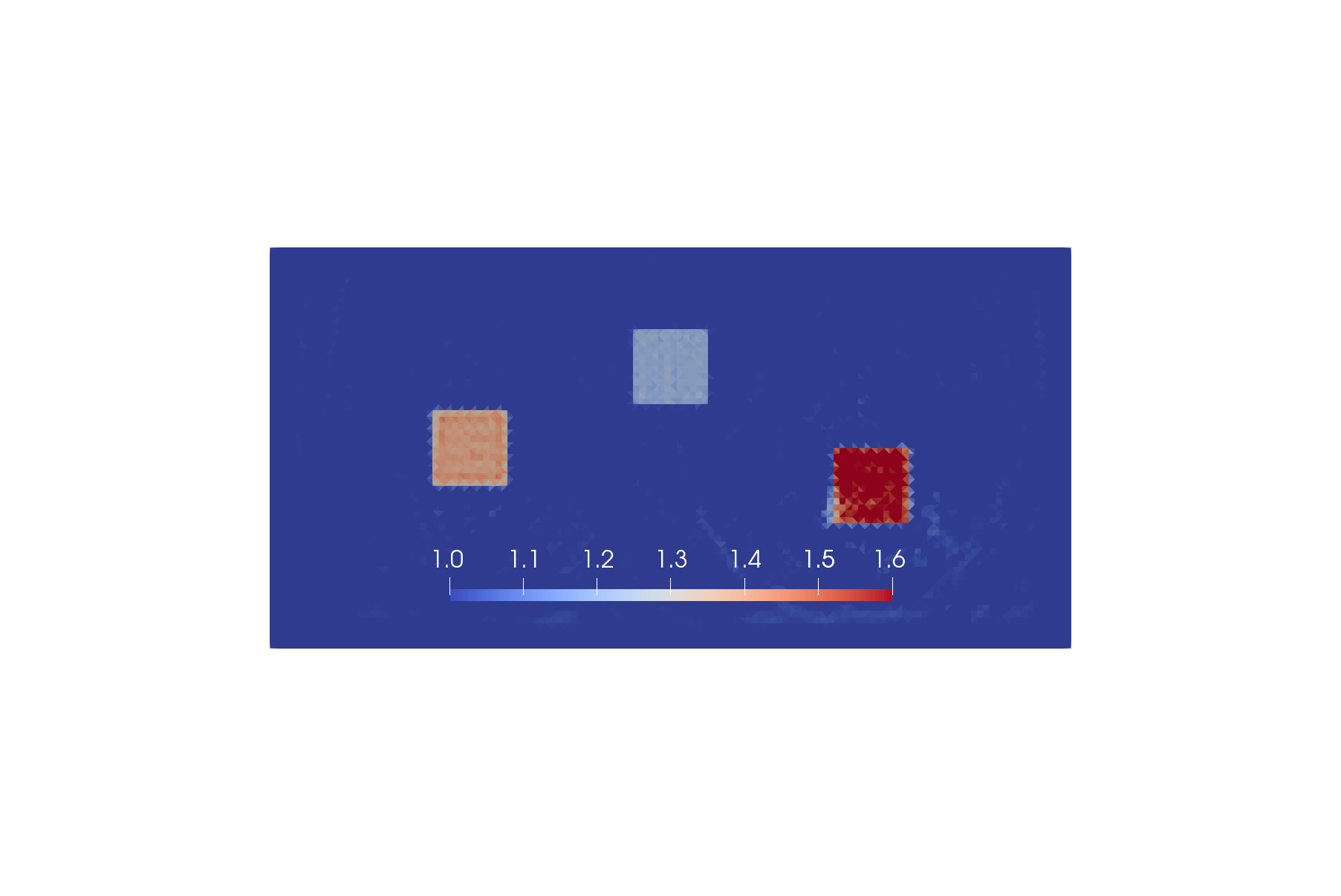}
	}
	\caption{Reconstruction in the numerical experiment.} 
	\label{fig:experiment}
\end{figure}

\subsection*{Acknowledgments} The authors are very thankful to Joachim Rehberg (WIAS Berlin) for the insightful discussion and his valuable hints concerning Sobolev spaces with partially vanishing traces.

\bibliographystyle{plainurl}
\bibliography{references}

\begin{thebibliography}{10}

\bibitem{ammann23}
Luis Ammann and Irwin Yousept.
\newblock Acoustic {F}ull {W}aveform {I}nversion via {O}ptimal {C}ontrol:
  {F}irst- and {S}econd-{O}rder {A}nalysis.
\newblock {\em SIAM J. Control Optim.}, 61(4):2468--2496, 2023.
\newblock \href {https://doi.org/10.1137/22M1480045}
  {\path{doi:10.1137/22M1480045}}.

\bibitem{ammann24}
Luis Ammann and Irwin Yousept.
\newblock Analysis of the {SQP} method for hyperbolic {PDE}-constrained
  optimization in acoustic full waveform inversion, 2024.
\newblock Preprint.
\newblock \href {https://arxiv.org/abs/2405.05158} {\path{arXiv:2405.05158}}.

\bibitem{brenner08}
Susanne~C. Brenner and L.~Ridgway Scott.
\newblock {\em The mathematical theory of finite element methods}, volume~15 of
  {\em Texts in Applied Mathematics}.
\newblock Springer, New York, third edition, 2008.
\newblock \href {https://doi.org/10.1007/978-0-387-75934-0}
  {\path{doi:10.1007/978-0-387-75934-0}}.

\bibitem{casastroeltzsch02}
Eduardo Casas and Fredi Tr\"oltzsch.
\newblock Error estimates for the finite-element approximation of a semilinear
  elliptic control problem.
\newblock {\em Control Cybernet.}, 31(3):695--712, 2002.
\newblock Well-posedness in optimization and related topics (Warsaw, 2001).

\bibitem{ciarlet02}
Philippe~G. Ciarlet.
\newblock {\em The Finite Element Method for Elliptic Problems}.
\newblock Society for Industrial and Applied Mathematics, 2002.
\newblock \href {https://doi.org/10.1137/1.9780898719208}
  {\path{doi:10.1137/1.9780898719208}}.

\bibitem{clark87}
Dean~S. Clark.
\newblock Short proof of a discrete {G}ronwall inequality.
\newblock {\em Discrete Appl. Math.}, 16(3):279--281, 1987.
\newblock \href {https://doi.org/10.1016/0166-218X(87)90064-3}
  {\path{doi:10.1016/0166-218X(87)90064-3}}.

\bibitem{egert17}
Moritz Egert and Patrick Tolksdorf.
\newblock Characterizations of {S}obolev functions that vanish on a part of the
  boundary.
\newblock {\em Discrete Contin. Dyn. Syst. Ser. S}, 10(4):729--743, 2017.
\newblock \href {https://doi.org/10.3934/dcdss.2017037}
  {\path{doi:10.3934/dcdss.2017037}}.

\bibitem{ern04}
Alexandre Ern and Jean-Luc Guermond.
\newblock {\em Theory and practice of finite elements}, volume 159 of {\em
  Applied Mathematical Sciences}.
\newblock Springer-Verlag, New York, 2004.
\newblock \href {https://doi.org/10.1007/978-1-4757-4355-5}
  {\path{doi:10.1007/978-1-4757-4355-5}}.

\bibitem{hou13}
Tianliang Hou.
\newblock Error estimates and superconvergence of semidiscrete mixed methods
  for optimal control problems governed by hyperbolic equations.
\newblock {\em Numerical Analysis and Applications}, 5, 10 2012.
\newblock \href {https://doi.org/10.1134/S1995423912040076}
  {\path{doi:10.1134/S1995423912040076}}.

\bibitem{logg12}
Anders Logg, Kent-Andre Mardal, and Garth~N. Wells, editors.
\newblock {\em Automated solution of differential equations by the finite
  element method}, volume~84 of {\em Lecture Notes in Computational Science and
  Engineering}.
\newblock Springer, Heidelberg, 2012.
\newblock The FEniCS book.
\newblock \href {https://doi.org/10.1007/978-3-642-23099-8}
  {\path{doi:10.1007/978-3-642-23099-8}}.

\bibitem{logg10}
Anders Logg and Garth~N. Wells.
\newblock D{OLFIN}: automated finite element computing.
\newblock {\em ACM Trans. Math. Software}, 37(2):Art. 20, 28, 2010.
\newblock \href {https://doi.org/10.1145/1731022.1731030}
  {\path{doi:10.1145/1731022.1731030}}.

\bibitem{peralta22}
Gilbert Peralta and Karl Kunisch.
\newblock Mixed and hybrid {P}etrov-{G}alerkin finite element discretization
  for optimal control of the wave equation.
\newblock {\em Numer. Math.}, 150(2):591--627, 2022.
\newblock \href {https://doi.org/10.1007/s00211-021-01258-9}
  {\path{doi:10.1007/s00211-021-01258-9}}.

\bibitem{yee66}
Kane Yee.
\newblock Numerical solution of initial boundary value problems involving
  maxwell's equations in isotropic media.
\newblock {\em IEEE Transactions on Antennas and Propagation}, 14(3):302--307,
  1966.
\newblock \href {https://doi.org/10.1109/TAP.1966.1138693}
  {\path{doi:10.1109/TAP.1966.1138693}}.

\end{thebibliography}

\end{document}